\documentclass[12pt,reqno]{amsart}
\usepackage{amssymb}
\usepackage{enumerate}
\usepackage{amsthm}
\usepackage{tikz-cd}
\usepackage{amsmath}
\usepackage{hyperref}

\numberwithin{equation}{section}

\newtheorem{thm}{Theorem}[section]

\newtheorem{prop}[thm]{Proposition}

\newtheorem*{thm*}{Theorem}

\theoremstyle{definition}
\newtheorem{defi}[thm]{Definition}
\newtheorem{asump}[thm]{Assumption}
\newtheorem{xrem}[thm]{Remark}
\newtheorem{exm}[thm]{Example}
\newtheorem{question}[thm]{Question}

\DeclareMathOperator{\rank}{{rank}}

\DeclareMathOperator{\Aut}{{Aut}}
\DeclareMathOperator{\Gal}{{Gal}}
\DeclareMathOperator{\red}{{red}}

\DeclareMathOperator{\mult}{{mult}}

\DeclareMathOperator{\Bl}{{Bl}}

\begin{document}
\baselineskip=16pt

\title[Seshadri constants of parabolic vector bundles]{Seshadri constants of parabolic 
vector bundles}

\author[I. Biswas]{Indranil Biswas}

\address{School of Mathematics, Tata Institute of Fundamental
Research, Homi Bhabha Road, Mumbai 400005, India}

\email{indranil@math.tifr.res.in}

\author[K. Hanumanthu]{Krishna Hanumanthu}

\address{Chennai Mathematical Institute, H1 SIPCOT IT Park, Siruseri, Kelambakkam 603103, 
India}

\email{krishna@cmi.ac.in}

\author[S. Misra]{Snehajit Misra}

\address{Chennai Mathematical Institute, H1 SIPCOT IT Park, Siruseri, Kelambakkam 603103, 
India}

\email{snehajitm@cmi.ac.in}

\author[N. Ray]{Nabanita Ray}

\address{Chennai Mathematical Institute, H1 SIPCOT IT Park, Siruseri, Kelambakkam 603103, 
India}

\email{nabanitar@cmi.ac.in}



\date{}

\begin{abstract}
Let $X$ be a complex projective variety, and let $E_{\ast}$ be a parabolic vector bundle on 
$X$ with parabolic structure on a divisor. We introduce the notion of \textit{parabolic Seshadri constants} of $E_{\ast}$. 
It is
shown that these constants are analogous to the classical Seshadri constants of vector 
bundles, in particular they have parallel definitions and properties. We prove a 
Seshadri criterion for parabolic ampleness of $E_{\ast}$ in terms of the parabolic Seshadri 
constants of it. We also compute parabolic Seshadri constants for symmetric powers and tensor 
products of parabolic vector bundles.
\end{abstract}

\maketitle

\tableofcontents

\section{Introduction}\label{intro}

Let $X$ be a complex projective variety, and let $L$ be a nef line bundle on
$X$. For a point $x \in X$, the \textit{Seshadri constant} of $L$ at $x$,
denoted by $\varepsilon(X,\,L,x)$, is defined
to be
$$\varepsilon(X,\,L,x)\,:=\, \inf\limits_{\substack{x \in C}}\, \frac{L\cdot
C}{{\rm mult}_{x}C}\, ,$$
where the infimum is taken over all irreducible and reduced curves $C\,\subset\, X$
passing through $x$; here 
$L\cdot C$ denotes the intersection number (equivalently, the degree of $L$
restricted to $C$), while ${\rm mult}_x C$ denotes the 
multiplicity of the curve $C$ at $x$. 

Alternatively, we can define Seshadri constants as follows. 
Let $\pi\,:\, \widetilde{X}\,\longrightarrow\, X$ be the blow up of $X$ at $x$, and
let $A$ denote the corresponding exceptional divisor. Then
$$\varepsilon(X,\,L,x)\, =\,
{\rm sup}~\{\lambda \,\ge\, 0 \ \big\vert\ \, \pi^{\ast}(L)-\lambda A\,\ ~\rm{is~ nef}\}.$$

This constant was first introduced by Demailly in \cite{Dem90} while he was studying 
problems related to generation by jets and Fujita's Conjecture. The name originated from an 
ampleness criterion of Seshadri \cite[Theorem I.7.1]{Har70}. Seshadri constant of line 
bundles blossomed into a very active area of research, especially in connection with 
the positivity properties of the line bundle $L$. Seshadri constants for vector bundles of 
arbitrary rank have also been studied. They were defined in \cite{Hac00} and studied 
in depth in \cite{FM21}. See \cite{HMP10,BHN20, DKS22,MR22} for some results on Seshadri 
constants of vector bundles.

A \textit{parabolic vector bundle} on a curve is a vector bundle $E$ with some additional 
structures given by filtrations of certain fibers of $E$ along with weights attached to these 
filtrations (see Definition \ref{defn-parabolic}). We recall that parabolic vector bundles
were introduced by Seshadri (see \cite{Ses77}) while studying certain representations of the 
fundamental groups of punctured Riemann surfaces; they were further studied
by Mehta and Seshadri in \cite{MS80}. The notion of parabolic vector bundle was generalized
to higher dimensions by Maruyama and Yokogawa in \cite{MY92}.

There has been a lot of research on parabolic vector bundles, starting with a description 
of their moduli spaces in \cite{MS80,MY92}. In \cite{B97II}, certain orbifold bundles were 
associated to parabolic bundles (see also \cite{Bo1}, \cite{Bo2}). The notion of a 
\textit{ramified} principal ${\rm GL}_n(\mathbb{C})$-bundle, which generalizes the notion 
of principal ${\rm GL}_n(\mathbb{C})$-bundles, was introduced in \cite{BBN03}. It was shown 
there that parabolic vector bundles can be viewed as ramified ${\rm GL}_n ({\mathbb 
C})$-bundles. Using this identification, \cite{BL11} constructed a projectivization of a 
parabolic vector bundle as well as the tautological line bundle on the projectivization.

Many positivity properties, such as ampleness of line bundles, have been generalized to 
vector bundles. Some of these notions have also been generalized to the
context of parabolic vector bundles. 
For example, the notion of ample parabolic bundles was defined in \cite{B97I}. The 
ampleness of a parabolic vector bundle is characterized by the ampleness of the 
tautological bundle on the projectivization associated to that parabolic bundle (see 
\cite{BL11}).

Proceeding further in this direction, it is natural to ask whether other positivity notions 
can also be generalized to the set-up of parabolic vector bundles. As remarked earlier, 
Seshadri constant is an important invariant in the study of positivity. In this paper, we 
generalize the notion of Seshadri constants to parabolic vector bundles on projective 
varieties. Crucial ingredients in our construction are orbifold bundles associated to the
parabolic bundles and projectivization of parabolic bundles. Many of our results and proofs 
are motivated by analogous results of Fulger and Murayama in \cite{FM21}.

In Section \ref{prelims}, we recall the definition and basic properties of parabolic vector 
bundles that we use. They are included for the sake of completeness and also for the 
convenience of the reader.

Let $X$ be an irreducible projective variety and 
let  $E_{\ast}$ be a  parabolic vector bundle on $X$. 
In Section \ref{defn-sc}, we give the definition of the parabolic Seshadri constant  $\varepsilon_{\ast}(E_{\ast},\,x)$ 
of  $E_{\ast}$ at any point $x \in X$; see Definition \ref{main-defn}. 
This definition makes use of the notion of projectivization of a parabolic bundle; see Section \ref{proj} for a description of 
this construction.  Alternately, one can also use the notion of 
orbifold bundles (described in Section \ref{orbifold}) to define parabolic Seshadri constants.

Different alternative characterizations of parabolic Seshadri constants are given in Theorem \ref{corl2.1} and 
Theorem \ref{thm2.2}.

We also prove an analogue of the Seshadri criterion for ampleness for parabolic
bundles (see Theorem \ref{seshadri-crit}).  We state this criterion here for the convenience of the reader. See Section \ref{defn-sc} for the notation. 

\begin{thm*}[Theorem \ref{seshadri-crit}]
Let $E_{\ast}$ be a parabolic vector bundle on a smooth irreducible projective variety 
$X$ such that the numerical class $\xi' \,\equiv\, \mathcal{O}_{\mathbb{P}(E')}\bigl(N(E_*)\bigr)$ is $\tau$-ample. Then
$E_{\ast}$ is parabolic ample if and only if $$\inf\limits_{x\in 
X}\varepsilon_{\ast}(E_{\ast},\,x)\,>\, 0,$$ where the infimum is taken over all points of $X$.
\end{thm*}

An 
upper bound for parabolic Seshadri constants is given in Theorem \ref{bound}.

In Section \ref{properties-sc}, we prove several properties of parabolic Seshadri constants. We define
the notion of multi-point Seshadri constants for the usual vector bundles and relate it to
the parabolic Seshadri constants.

One of the results proved in this section describes
how parabolic Seshadri constants can be computed using 
restriction to curves.
\begin{thm*}[Theorem \ref{thm3.2}]
Let $E_{\ast}$ be a parabolic nef vector bundle on a smooth irreducible complex projective
variety $X$, and let $E'\,\longrightarrow\, Y$ be the corresponding orbifold bundle over
$Y$ (see Section \ref{orbifold}). Then 
\begin{align*}
\varepsilon_{\ast}(E_{\ast},\,x) \,\,=\,\, N(E_{\ast})\cdot\inf\limits_{C\subset Y}
\left\{\frac{\mu_{\min}(\nu^{\ast}E')}{\sum\limits_{y\in \gamma^{-1}(x)}\mult_y C}\right\},
\end{align*}
where the infimum is taken over all irreducible curves $C\,\subset\, Y$ such that
$C\bigcap \gamma^{-1}(x) \,\neq \,\emptyset$, and $\nu \,:\, \overline{C}
\,\longrightarrow\, C$ is the normalization map.
\end{thm*}

We also compute parabolic Seshadri constants of symmetric powers of
parabolic vector bundles (see Theorem \ref{thm-sym}) as well as tensor products of parabolic vector
bundles (see Theorem \ref{thm-tensor}). These results are stated below.

\begin{thm*}[Theorem \ref{thm-sym}]
Let $E_{\ast}$ be a parabolic nef vector bundle on a smooth irreducible complex
projective variety $X$. Then for any positive integer $m$
$$\varepsilon_{\ast}(S^m(E_{\ast}),\,x)\,=\, m\varepsilon_{\ast}(E_{\ast},\,x)$$ for every
point $x\,\in\, X$.
\end{thm*}

\begin{thm*}[Theorem \ref{thm-tensor}]
Let $E_{\ast}$ and $F_{\ast}$ be two parabolic nef vector bundles on a smooth irreducible
complex projective variety $X$ having a common parabolic divisor $D\,\subset\, X$. Then for
every point $x\,\in\, X$,
\begin{align*}
\varepsilon_{\ast}(E_{\ast}\otimes F_{\ast},\,x)\,\,\geq\,\, N(E_{\ast}\otimes F_{\ast})
\cdot\left\{\frac{\varepsilon_{\ast}(E_{\ast},\,x)}{N(E_{\ast})}\,+\,
\frac{\varepsilon_{\ast}(F_{\ast},\,x)}{N(F_{\ast})}\right\}.
\end{align*}
\end{thm*}

Some examples and questions are mentioned in Section \ref{examples}. 

We work throughout over the field $\mathbb{C}$ of complex numbers. The field of real 
numbers is denoted by $\mathbb{R}$.

\section{Preliminaries}\label{prelims}

In this section, we briefly recall the definitions, 
properties and constructions associated to parabolic vector bundles; more details can be 
found in \cite{B97I,B97II,B98,BS99,BBN03,BL11}.

\subsection{Parabolic sheaves}\label{defn-parabolic-sheaves}

We start by defining the parabolic vector bundles. 

\begin{defi} \label{defn-parabolic}
Let $X$ be a connected smooth complex projective variety of dimension $d$, and let $D \,\subset\, X$ be an effective divisor on $X$.
\begin{enumerate} 
\item Let $E$ be a torsion-free coherent $\mathcal{O}_X$-module. A \textit{quasi-parabolic structure} on $E$ with respect to $D$ is a filtration by $\mathcal{O}_X$-coherent subsheaves
\begin{equation}\label{ef}
E \,=\, \mathcal{F}_1(E) \,\supset\, \mathcal{F}_2(E) \,\supset\, \cdots\, \supset\, \mathcal{F}_l(E) \,\supset\, \mathcal{F}_{l+1}(E) \,=\, E(-D),
\end{equation}
where $E(-D) \,:= \,E\otimes_{\mathcal{O}_X}\mathcal{O}_X(-D)$. The above integer $l$ is
called the \textit{length} of the filtration.

\item A \textit{parabolic structure} on $E$ with respect to $D$ is a quasi-parabolic 
structure as above together with a system of \textit{weights} $\{ 
\alpha_1,\,\alpha_2,\,\cdots,\,\alpha_l\}$, where each $\alpha_i$ is a real number such 
that $0 \,\leq \,\alpha_1\,<\,\alpha_2\,<\,\cdots \,<\,\alpha_{l-1} \,<\,\alpha_l\,<\,1$.
The numbers $\{\alpha_i\}_{i=1}^l$ are called \textit{parabolic weights} and we say that 
$\alpha_i$ is attached to $\mathcal{F}_i(E)$.

\item A \textit{parabolic sheaf} is a torsion-free coherent $\mathcal{O}_X$-module $E$ with 
a parabolic structure.

We denote a parabolic sheaf by the triple $\bigl( E ,\, 
\mathcal{F}_{\ast},\,\alpha_{\ast}\bigr)$, where $\mathcal{F}_{\ast}$ denotes the above 
filtration and $\alpha_{\ast}$ denotes the system of parabolic weights. When 
$\mathcal{F}_{\ast}$ and $\alpha_{\ast}$ are clear from the context, we denote
$\bigl( E ,\, \mathcal{F}_{\ast},\,\alpha_{\ast}\bigr)$ by $E_{\ast}$.

\item If $E$ is a vector bundle on $X$, then $E_{\ast}$ is called a \textit{parabolic vector bundle with parabolic divisor $D$}. The \textit{rank} of a parabolic vector bundle is simply the rank of the underlying vector bundle $E$. 
\end{enumerate}
\end{defi}
For any parabolic sheaf $E_{\ast}$ defined as above, and for any $t\,\in\, \mathbb{R}$,
define the following filtration $\{E_t\}_{t\in \mathbb{R}}$
of coherent sheaves parametrized by $\mathbb{R}$:
\begin{align*}
 E_t\, =\, \mathcal{F}_i(E)(-[t]D),
\end{align*}
where $[t]$ is the integral part of $t$ and $\alpha_{i-1} \,<\, t-[t] \,\leq\, \alpha_i$ 
with $\alpha_0 \,=\, \alpha_l-1$ and $\alpha_{l+1} \,=\, 1$. The filtration $\{E_t\}_{t\in 
\mathbb{R}}$ evidently determines the parabolic structure $\bigl( E ,\, 
\mathcal{F}_{\ast},\,\alpha_{\ast}\bigr)$ uniquely. Note that any coherent subsheaf $M\, 
\subset\,E$ has an induced parabolic structure such that the corresponding filtration 
$\{M_t\}_{t\in \mathbb{R}}$ is defined by $M_t \,:= \,(E_{t-[t]}\bigcap M)([t]D)$.

Consider the decomposition 
\begin{align*}
D \,= \,\sum\limits_{j=1}^{n} b_jD_j,
\end{align*}
where every $D_j$ is a reduced irreducible divisor, and $b_j \,\geq\, 1$. Let 
\begin{align*}
f_j \,:\, b_jD_j\,\longrightarrow\, X
\end{align*}
denote the inclusion map of the subscheme $b_jD_j$. For each $1\,\leq\, j\,\leq\, n$,
choose a filtration
\begin{align}\label{seq1}
 0 \,=\, F_{l_j+1}^j\, \subset\, F_{l_j}^j\,\subset\, F_{l_j-1}^j\,\subset\, \cdots \,\subset\, F_1^j = f_j^{\ast}E\, .
\end{align}
Fix real numbers $\alpha_k^j$,\, $1\,\leq\, k\,\leq \, l_j+1$, such that
\begin{align*}
1 \,=\, \alpha_{l_j+1}^j \,> \,\alpha_{l_j}^j \,>\, \alpha_{l_j-1}^j\, >\, \cdots \,>\, \alpha_2^j \,> \,\alpha_1^j \,\geq\, 0.
\end{align*}
For every $1\,\leq\, j\,\leq\, n$ and $1\,\leq\, k\,\leq \, l_j+1$, define the coherent
subsheaf $\overline{F}_j^i\, \subset\, E$
using the following short exact sequences:
\begin{align*}
 0 \,\longrightarrow \,\overline{F}_k^j \,\longrightarrow\, E \,\longrightarrow\, (f_j^{\ast}E)/F_k^j \,\longrightarrow\, 0.
\end{align*}
For $1\,\leq\, j\,\leq\, n$ and $0 \,\leq\, t \,\leq\, 1$, let $l_t^j$ be the smallest number in the set of integers 
$$
\bigl\{ k \,\in\, \{1,\,2,\,\cdots,\,l_j+1\} \,\,\big\vert\,\, \alpha_k^j \,\geq\, t\bigr\}.
$$
Finally, set
\begin{align}\label{z1}
E_t \,:=\, \bigcap_{j=1}^n \overline{F}_{l_{t}}^j \, \subseteq\, E.
\end{align}
The filtration $\{ E_t\}_{t\in \mathbb{R}}$ in \eqref{z1} defines a parabolic structure on $E$. It is
straightforward to check that all parabolic structures on $E$, with $D$ as the parabolic divisor, arise 
this way.

Let
\begin{equation}\label{ep}
\phi\,:\,X\setminus D\,\longrightarrow\,X
\end{equation}
denote the inclusion map of the complement
of $D$. Let $E_{\ast}$ and $W_{\ast}$ be two parabolic sheaves on $X$, with
$D$ as the parabolic divisor, such that $E_0\big\vert_{X\setminus D}$ and
$W_0\big\vert_{X\setminus D}$ are locally free. For any $c\,\in\, \mathbb{R}$,
define $M_c$ to be the subsheaf of $\phi_{\ast}\phi^{\ast}(E\otimes W)$ generated by all
$E_s\otimes W_t$ with $s+t \,\geq\, c$, where $\phi$ is the map in \eqref{ep}.
The parabolic sheaf given by the filtration $\bigl\{M_c\bigr\}_{c\in \mathbb{R}}$ is 
called the {\it parabolic tensor product} of $E_{\ast}$ and $W_{\ast}$, and it is denoted 
by $E_{\ast}\otimes W_{\ast}$. The \textit{parabolic $m$-fold symmetric product} 
$S^m(E_{\ast})$ is the invariant subsheaf of the $m$-fold parabolic tensor product of 
$E_{\ast}$ for the natural action of the permutation group for the factors of the tensor 
product. The underlying sheaf of the parabolic sheaf $S^m(E_{\ast})$ will be 
denoted by $S^m(E_{\ast})_0$.

\begin{defi}[{\cite[Definition 2.3]{B97I}}]
A parabolic bundle $E_{\ast}$ is called \textit{parabolic ample} if for any coherent
sheaf $F$ on $X$ there is an integer $m_0$ such that for any $m\,\geq\, m_0$, the tensor
product $F \otimes S^m(E_{\ast})_0$ is globally generated.
\end{defi}

\begin{defi}[{\cite[Definition 3.2]{BS99}}]
A parabolic vector bundle $E_{\ast}$ is called \textit{parabolic nef} if there is
an ample line bundle $L$ over $X$ such that $S^m(E_{\ast})\otimes L$ is parabolic
ample for every positive integer $m$.
\end{defi}

\subsection{Semistability of parabolic bundles}\label{semistability}

We fix an ample line bundle $L$ on $X$. For any coherent $\mathcal{O}_X$-module $E$, the 
\textit{degree of $E$ with respect to $L$} is defined to be
\begin{equation}\label{ed}
\deg(E) \,:= \,\bigl(c_1(E) \cup c_1(L)^{d-1}\bigr) \cap [X] \,\in\, \mathbb{Z}
\end{equation}
(see \cite{MY92}). The {\it parabolic degree of $E_{\ast}$ with respect to $L$}, denoted by
${\rm par\_deg}(E_{\ast})$, is defined as follows:
$$
{\rm par\_deg} (E_{\ast})\,:=\, \int_{-1}^{0}\deg(E_t) \,dt\, \in\, \mathbb R .
$$
The {\it parabolic slope} of $E_{\ast}$, denoted by ${\rm par}\_\mu(E_{\ast})$, is the
quotient $\frac{{\rm par\_deg}(E_{\ast})}{{\rm rank}(E)}\,\in\, \mathbb R$.

\begin{defi}
A parabolic sheaf $E_{\ast}$ is called \textit{parabolic semistable} (respectively, 
\textit{parabolic stable}) if for every 
subsheaf $M$ of $E$ such that $ 0 \,<\,\rank(M) \,< \,\rank(E)$, the following inequality
is satisfied:
$${\rm par}\_\mu(M_{\ast}) \,\leq\, {\rm par}\_\mu(E_{\ast})\ \ \ (\text{respectively, }\,~
{\rm par}\_\mu(M_{\ast}) \,< \, {\rm par}\_\mu(E_{\ast}) ).$$
 \end{defi}
 
\begin{asump}\label{assumption}
Henceforth we will always impose the following four conditions on the parabolic bundles
$E_{\ast}$, with parabolic divisor $D$, that we will consider:
\begin{enumerate}[(a)]
 \item{ The parabolic divisor $D \,=\, \sum\limits_{i=1}^n b_iD_i$ is a normal crossing 
divisor, i.e., all $b_i\,=\,1$ and $D_i$ are smooth divisors and they intersect 
transversally.}
 
\item{ All $F_j^i$ on $D_i$ in sequence (\ref{seq1}) are subbundles of $f_i^{\ast}E$ for every $i$.}

\item{ All the weights $\alpha_j^i$ are rational numbers; so $\alpha_j^i \,=\, m_j^i/N$, where
$N$ is some fixed integer and $m_j^i \,\in\, \{ 0,\,1,\,\cdots,\, N-1\}$.}

\item{ Every point $x$ of $D$ has a neighborhood $U_x\, \subset\, X$, and
a decomposition of $E\big\vert_{U_x}$ into a direct sum of line bundles, such that
the filtration of all $E\big\vert_{U_x\cap D_i}$, $1\, \leq\, i\,\leq\, n$, are
constructed using the decomposition. To explain this condition in detail, let $D_{i_1},\, \cdots,\, D_{i_x}$ be
the irreducible components of $D\,=\, \sum_{i=1}^n D_i$ that contain the point $x$. For every $k\, \in\, {i_1,\, \cdots,\,
i_x}$, consider the filtration $\{F_j^k\}_{j=1}^{l_k}$ of $f_k^{\ast}E \,=\, E\big\vert_{D_k}$ in sequence \eqref{seq1} (recall
that $b_k\,=\,1$ by (a)). The condition says that there is a holomorphic decomposition into a direct sum of line bundles
$$
E\big\vert_{U_x}\, =\, \bigoplus_{\beta =1}^{r_E} L_\beta,
$$
where $r_E\,=\, \text{rank}(E)$, such that every $F_j^k\big\vert_{U_x\cap D_k}$ considered above
is of the form
$$
F_j^k\big\vert_{U_x\cap D_k}\,=\, \bigoplus_{\alpha =1}^{r_{k,j}} L_{d_\alpha}\big\vert_{U_x\cap D_k},
$$
where $r_{k,j}\,=\, \text{rank}(F_j^k)$ and $1\, \leq\, d_1 \, <\, d_2\, <\, \cdots\, <\, d_{r_{k,j}}\, \leq\, r_E$.}
\end{enumerate}
\end{asump}
 
\subsection{Orbifold bundles}\label{orbifold}

Let $Y$ be a smooth complex projective variety; its group
of algebraic automorphisms will be denoted by $\Aut(Y)$. Let $\Gamma$ be a finite group and $\psi \,:\, \Gamma
\,\longrightarrow\, \Aut(Y)$ a homomorphism giving an action of $\Gamma$ on $Y$. 

\begin{defi}
An \textit{orbifold bundle} on $Y$, with $\Gamma$ as the \textit{orbifold group}, is a 
vector bundle $V$ on $Y$ together with a lift of the action of $\Gamma$ on $Y$ to $V$, i.e.,
$\Gamma$ acts on the total space of $V$ such that the action of any $g\,\in\, \Gamma$
gives a vector bundle isomorphism between $V$ and $\psi(g^{-1})^{\ast}V$. A subsheaf $F$ of
an orbifold bundle $V$ is called an \textit{orbifold subsheaf} if the 
action of $\Gamma$ on $V$ preserves $F$. A homomorphism $V\, \longrightarrow\, V'$ between orbifold bundles
is a homomorphism $V\, \longrightarrow\, V'$ between the vector bundles that commutes with the actions of
$\Gamma$ on $V$ and $V'$. 
\end{defi}

Let $\widetilde{L}$ be an orbifold line bundle on $Y$ which is also ample. Then we can 
define the degree of any coherent sheaf on $Y$ using $\widetilde{L}$ as done in \eqref{ed}.

\begin{defi}
An orbifold bundle $V$ on $Y$ is called \textit{orbifold semistable}
(respectively, \textit{orbifold stable}) if for any orbifold subsheaf
$F$ of $V$ with $0 \,<\,\rank(F) \,< \,\rank(E)$ the following inequality holds:
$$
\frac{\deg(F)}{\rank(F)} \,\leq\, \frac{\deg(V)}{\rank(V)}\ \ ~ \left( \rm respectively,\,\
\frac{\deg(F)}{\rank(F)}
\,<\, \frac{\deg(V)}{\rank(V)}\right).
$$
\end{defi}

Now let $X$ be a smooth projective complex variety, and let $D$ be an effective divisor on 
$X$ satisfying Assumption \ref{assumption}(a). Fix an integer $N\, \geq\, 2$. From
the Covering Lemma of Y. Kawamata, \cite[Theorem 17]{K81} and \cite[Theorem 1.1.1]{KMM}, we
know that there is a connected smooth projective variety $Y$ and a (ramified)
Galois covering morphism 
\begin{align}\label{seq2}
\gamma \,:\, Y \,\longrightarrow\, X,
\end{align}
with Galois group $\Gamma \,=\, \Gal\bigl(K(Y)/K(X)\bigr)$, such that
$\widetilde{D}\,:=\, (\gamma^{\ast}D)_{\red}$ is a normal crossing divisor on $Y$, and
\begin{equation}\label{em}
\gamma^{\ast}D_i \,=\, k_iN(\gamma^{\ast}D_i)_{\red}
\end{equation}
for all $1\,\leq\, i\,\leq\, n$, where
$k_i$ are some positive integers.

Let $E_{\ast}$ be a parabolic bundle on $X$. Then using the Galois cover $\gamma \,:\, Y \,\longrightarrow \,X$ in \eqref{seq2} we can 
construct an orbifold bundle $V$ on $Y$ such that the parabolic bundle $E_{\ast}$ is recovered from it by taking $\Gamma$-invariants 
of the direct image of the twists of $V$ using the irreducible components of $\widetilde{D}$ (see \cite{B97II}, \cite{Bo1}, \cite{Bo2} for 
the explicit construction). Note that the same $\gamma$ is used for all parabolic bundles $E_*$ with the same parabolic divisor $D$.

For the convenience of the reader we will briefly recall the construction of an orbifold bundle on $Y$ from a parabolic bundle
on $X$ with parabolic structure on $D$. First let ${\mathcal L}_\ast$ be a parabolic line bundle on $X$ with parabolic structure on $D$.
So $\mathcal L$ is a line bundle on $X$, and for each $1\, \leq\, i\, \leq\, n$ we have
$0\, \leq\, \alpha^i\,=\, \frac{m^i_1}{N}\, <\, 1$ which is the
parabolic weight over the component $D_i$ of $D$; see Assumption \ref{assumption}(c). Then the orbifold line bundle
${\mathcal L}'$ on $Y$ corresponding to the parabolic line bundle ${\mathcal L}_\ast$ is 
$$
{\mathcal L}'\,=\, (\gamma^*{\mathcal L})\otimes \left(\bigotimes_{i=1}^n {\mathcal O}_Y(k_im^i_1(\gamma^{\ast}D_i)_{\red})\right)
$$
(see \eqref{em} for $k_i$); note that each $(\gamma^{\ast}D_i)_{\red}$ is preserved by the action of $\Gamma$ on $Y$ and hence the
above line bundle $$\bigotimes_{i=1}^n {\mathcal O}_Y(k_im^i_1(\gamma^{\ast}D_i)_{\red})$$ has a natural orbifold structure. The
orbifold structure on ${\mathcal L}'$ is given by the orbifold structures on 
$\gamma^*{\mathcal L}$ and $\bigotimes_{i=1}^n {\mathcal O}_Y(k_im^i_1(\gamma^{\ast}D_i)_{\red})$. The action of
$\Gamma$ on ${\mathcal L}'$ produces an action of $\Gamma$ on the direct image $\gamma_*{\mathcal L}'$. It is straightforward to
check that
$$
{\mathcal L}\,=\, (\gamma_*{\mathcal L}')^\Gamma \, \subset\,\gamma_*{\mathcal L}'.
$$
Note that this construction is local in the sense that we do not need $X$ or $Y$ to be a projective variety. Let $S\, \subset\, D$
denote the singular locus of $D$, and let
\begin{equation}\label{ei}
\iota\,\,:\,\, Y\setminus \gamma^{-1}(S)\, \,\hookrightarrow\,\, Y
\end{equation}
be the inclusion map. For a parabolic bundle $E_\ast$ on $X$, consider the restriction of it to the complement $X\setminus S$. It can be
locally expressed as direct sums of parabolic line bundles. Therefore, using the above construction of an orbifold line bundle from
a parabolic line bundle, and patching locally defined orbifold bundles, we
get an orbifold vector bundle $\mathcal{E}$ on $Y\setminus \gamma^{-1}(S)$. For the patching to work compatibly, note that if
$V_1$ and $V_2$ are two orbifold bundles on $\gamma^{-1}(U)$ for two decompositions of the parabolic bundle $E_\ast\big\vert_U$
into direct sum of parabolic line bundles, then the natural isomorphism of orbifold bundles
$$
\gamma^*\left(E\big\vert_{U\setminus (U\cap D)}\right)\,=\,
V_1\big\vert_{\gamma^{-1}(U)\cap (Y\setminus \gamma^{-1}(D))} \, \stackrel{\sim}{\longrightarrow}\,
V_2\big\vert_{\gamma^{-1}(U)\cap (Y\setminus \gamma^{-1}(D))}\,=\, \gamma^*\left(E\big\vert_{U\setminus (U\cap D)}\right)
$$
extends to an isomorphism $V_1\big\vert_{\gamma^{-1}(U)\cap (Y\setminus \gamma^{-1}(S))} \, \stackrel{\sim}{\longrightarrow}\,
V_2\big\vert_{\gamma^{-1}(U)\cap (Y\setminus \gamma^{-1}(S))}$. Since the codimension of
$\gamma^{-1}(S)\, \subset\, Y$ is two (if $\gamma^{-1}(S)$ is nonempty), it follows that the direct image $\iota_*{\mathcal E}$
is a coherent sheaf on $Y$, where $\iota$ is the map in \eqref{ei}. The conditions in Assumption \ref{assumption} ensure that
the coherent sheaf
$\iota_*{\mathcal E}$ is locally free. The action of $\Gamma$ on ${\mathcal E}$ produces an action of $\Gamma$ on $\iota_*{\mathcal E}$.
This orbifold bundle $\iota_*{\mathcal E}$ is the one associated to the parabolic bundle $E_\ast$.

Let $L$ be an ample line bundle on $X$. Consider $\widetilde{L}\,:=\, \gamma^{\ast}L$ which 
is an orbifold line bundle on $Y$. Since $\gamma$ is a finite morphism, and
$L$ is ample, the line bundle 
$\widetilde{L}$ is also ample. We will use $L$ (respectively, $\widetilde{L}$) to define the 
degree of a coherent sheaf on $X$ (respectively, $Y$).

The above mentioned correspondence between parabolic bundles and orbifold
bundles satisfies the following compatibility condition with semistability.

\begin{prop}[{\cite[Lemma 2.7, Lemma 3.16]{B97II}}]
The above orbifold bundle $V$ corresponding to the parabolic
bundle $E_{\ast}$ is orbifold semistable with respect to
$\widetilde{L}$ if and only if $E_{\ast}$ is parabolic semistable with respect to $L$.
Also, $V$ is semistable with respect to $\widetilde{L}$ in
the usual sense if and only if $E_{\ast}$ is parabolic semistable with respect to $L$.
\end{prop}

\begin{prop}[{\cite[Lemma 4.6]{B97I}, \cite[Lemma 2.18]{B98}, \cite[Proposition 
3.2]{BS99}}]\label{prop-na}
A parabolic bundle $E_{\ast}$ is parabolic ample (respectively, nef) if and only if
the corresponding orbifold bundle $V$ is ample (respectively, nef) as a vector
bundle (in the usual sense).
\end{prop}

Parabolic Chern classes, denoted by $c_i(E_{\ast})$, of a parabolic bundle $E_{\ast}$ have 
been introduced in \cite{B98} (see also \cite{IS}). Let $V$ be the associated orbifold bundle on $Y$ 
corresponding to a parabolic bundle $E_{\ast}$ (satisfying Assumption \ref{assumption}) for 
a suitable Galois covering map $\gamma \,:\, Y \,\longrightarrow\, X$. Then 
$$\gamma^{\ast}c_i(E_{\ast})\,=\, c_i(V)$$ for all $i\,\geq\, 0$ (see \cite[Section 3]{B98}).

\subsection{Parabolic bundles as ramified ${\rm GL}(n,{\mathbb C})$-bundles}\label{proj}

Let $X$ be a smooth complex projective variety, and let $D$ be a
normal crossing divisor on $X$. 
Let $$\phi \,:\, E_{{\rm GL}(n,\mathbb{C})} \,\longrightarrow\, X$$
be a \textit{ramified principal ${\rm GL}(n,\mathbb{C})$-bundle over $X$ with ramification
over $D$}. This means the total space $E_{{\rm GL}(n,\mathbb{C})}$ is a smooth complex
quasiprojective variety equipped with an algebraic right action of ${\rm GL}(n,\mathbb{C})$ 
$$f\,\,:\,\,E_{{\rm GL}(n,\mathbb{C})} \times {\rm GL}(n,\mathbb{C})
\,\,\longrightarrow\,\, E_{{\rm GL}(n,\mathbb{C})}$$ satisfying the following five conditions: 
 \begin{enumerate}
\item $\phi \circ f \,=\, \phi \circ p_1$, where $p_1$ is the natural projection of
$E_{{\rm GL}(n,\mathbb{C})} \times{\rm GL}(n,\mathbb{C})$ to $E_{{\rm GL}(n,\mathbb{C})}$,

\item for each point $x\,\in\, X$, the action of ${\rm GL}(n,\mathbb{C})$ on the reduced
fiber $\phi^{-1}(x)_{red}$ is transitive,

\item the restriction of $\phi$ to $\phi^{-1}(X-D)$ a principal ${\rm GL}(n,\mathbb{C})$-bundle
over $X-D$,

\item for each irreducible component $D_i \subset D$, the reduced inverse image 
$\phi^{-1}(D_i)_{red}$ is a smooth divisor and $$\widehat{D} \,:=\, \sum\limits_{i=1}^l 
\phi^{-1}(D_i)_{red}$$ is a normal crossing divisor on $E_{{\rm GL}(n,\mathbb{C})}$, and

\item for any point $x\,\in\, D$, and any point $z\,\in\, \phi^{-1}(x)$, the isotropy
subgroup $G_z \,\subset\, {\rm GL}(n,\mathbb{C})$, for the action of ${\rm GL}(n,\mathbb{C})$
on $E_{{\rm GL}(n,\mathbb{C})}$, is a finite group, and if $x$ is a smooth point of $D$, then
the natural action of $G_z$ on the quotient line
$T_zE_{{\rm GL}(n,\mathbb{C})}/T_z\phi^{-1}(D)_{\rm red}$ is faithful.
\end{enumerate}

Note that the last condition implies that the isotropy
subgroup $G_z \,\subset\, {\rm GL}(n,\mathbb{C})$ is actually a finite cyclic group.

There is a natural bijective correspondence between the complex vector bundles of rank $n$
on $X$ and the
principal ${\rm GL}(n,\mathbb{C})$-bundles on $X$. This bijection sends a
principal ${\rm GL}(n,\mathbb{C})$-bundle $F$ to the vector bundle
$F\times^{{\rm GL}(n,\mathbb{C})}{\mathbb C}^n$ associated to $F$ for the standard action of
${\rm GL}(n,\mathbb{C})$ on ${\mathbb C}^n$.
This correspondence extends to a
bijective correspondence between the ramified principal ${\rm GL}(n,\mathbb{C})$-bundles with 
ramification over $D$ and parabolic vector bundles of rank $n$ with $D$ as the parabolic 
divisor (see \cite[Theorem 1.1]{BBN03}).

We next recall the construction of projectivization of a parabolic bundle. 
Let $E_{\ast}$ be a parabolic vector bundle over $X$ of rank $n$. Let
$$\phi \,\,:\,\, E_{{\rm GL}(n,\mathbb{C})} \,\,\longrightarrow\,\, X$$
be the corresponding ramified principal ${\rm GL}(n,\mathbb{C})$-bundle with ramification
divisor $D$. Consider the standard action of ${\rm GL}(n,\mathbb{C})$ on $\mathbb{C}^n$; it
induces an action of ${\rm GL}(n,\mathbb{C})$ on the projective space $\mathbb{P}^{n-1}$.
The \textit{projectivization} of $E_{\ast}$, denoted by $\mathbb{P}(E_{\ast})$, is defined
to be the associated (ramified) fiber bundle
$$\mathbb{P}(E_{\ast}) \,:=\, E_{{\rm GL}(n,\mathbb{C})}\bigl(\mathbb{P}^{n-1}\bigr)
\,:= \, E_{{\rm GL}(n,\mathbb{C})}\times^{{\rm GL}(n,\mathbb{C})} \mathbb{P}^{n-1}
\,\longrightarrow\, X.$$
 
Take any point $x\,\in\, D$ and any $z\,\in\, \phi^{-1}(x)$. Let $G_z\,\subset\,
{\rm GL}(n,\mathbb{C})$ be the isotropy subgroup for $z$ for the action of
${\rm GL}(n,\mathbb{C})$ on $E_{{\rm GL}(n,\mathbb{C})}$. We recall that $G_z$ is a
finite group (by condition (5) above). Let $n_x$ be the order of $G_z$. Note that the order
of the group $G_z$ is independent of the choice of $z\,\in\, \phi^{-1}(x)$ because
${\rm GL}(n,\mathbb{C})$ acts transitively on $\phi^{-1}(x)$. The number of distinct integers
$n_x$ as $x$ varies over $D$ is finite (see \cite{BL11}). Let 
\begin{align}\label{seq3}
N(E_{\ast}) \,\,=\,\, \text{l.c.m.} \bigl\{n_x \,\,\big\vert\,\, x\,\in\, D\bigr\}
\end{align}
be the least common multiple of all these finitely many positive integers $n_x$. 

For any point $y\,\in\, \mathbb{P}^{n-1}$, let $H_y\,\subset\,{\rm GL}(n,\mathbb{C})$ be the
isotropy subgroup for the natural action of ${\rm GL}(n,\mathbb{C})$ on $\mathbb{P}^{n-1}$;
so $H_y$ is a maximal parabolic subgroup of ${\rm GL}(n,\mathbb{C})$. The group $H_y$
then acts on the fiber of $\mathcal{O}_{\mathbb{P}^{n-1}}(1) \,\longrightarrow\,
\mathbb{P}^{n-1}$ over the point $y$. From the definition of $N(E_{\ast})$ in
\eqref{seq3} it follows immediately that for any $z\,\in\, \phi^{-1}(D)$ and any $y\,\in\,
\mathbb{P}^{n-1}$, the subgroup $G_z \bigcap H_y \,\subset\, {\rm GL}
(n,\mathbb{C})$ acts trivially on the fiber of the line bundle
$\mathcal{O}_{\mathbb{P}^{n-1}}\bigl(N(E_{\ast})\bigr)\,:=\,
\mathcal{O}_{\mathbb{P}^{n-1}} (1)^{\otimes N(E_{\ast})}$ over the point $y$.

Consider the action of ${\rm GL}(n,\mathbb{C})$ on the total space of
$\mathcal{O}_{\mathbb{P}^{n-1}}\bigl(N(E_{\ast})\bigr)$ constructed using the
standard action of ${\rm GL}(n,\mathbb{C})$ on $\mathbb{C}^n$.
Let
$$E_{{\rm GL}(n,\mathbb{C})}\bigl(\mathcal{O}_{\mathbb{P}^{n-1}}
\bigl(N(E_{\ast})\bigr)\bigr) \,:=\,E_{{\rm GL}(n,\mathbb{C})}\times^{{\rm GL}(n,\mathbb{C})}
\mathcal{O}_{\mathbb{P}^{n-1}}\bigl(N(E_{\ast})\bigr)\, \longrightarrow\, X$$
be the associated fiber bundle.
As the natural projection $\mathcal{O}_{\mathbb{P}^{n-1}}\bigl(N(E_{\ast})\bigr)
\,\longrightarrow \,\mathbb{P}^{n-1}$ intertwines the actions of
${\rm GL}(n,\mathbb{C})$ on $\mathcal{O}_{\mathbb{P}^{n-1}}\bigl(N(E_{\ast})\bigr)$
and $\mathbb{P}^{n-1}$, it produces a projection
\begin{equation}\label{evb}
E_{{\rm GL}(n,\mathbb{C})}\bigl(\mathcal{O}_{\mathbb{P}^{n-1}}\bigl(N(E_{\ast})\bigr)\bigr)
\,\longrightarrow\, E_{{\rm GL}(n,\mathbb{C})}\bigl(\mathbb{P}^{n-1}\bigr)
\,=\,\mathbb{P}(E_{\ast}).
\end{equation}
Using the above observation that $G_z\bigcap H_y$ acts trivially on the fibers of
$\mathcal{O}_{\mathbb{P}^{n-1}}\bigl(N(E_{\ast})\bigr)$ over $y$
it follows easily that the projection in \eqref{evb} makes
$E_{{\rm GL}(n,\mathbb{C})}\bigl(\mathcal{O}_{\mathbb{P}^{n-1}}\bigl(N(E_{\ast})\bigr)\bigr)$
an algebraic line bundle over the projectivization $\mathbb{P}(E_{\ast})$. We denote
this line bundle in \eqref{evb} by $\mathcal{O}_{\mathbb{P}(E_{\ast})}(1)$. 

Let $E' \,\longrightarrow\, Y$ be the orbifold bundle over $Y$
corresponding to the parabolic vector bundle $E_*$, where $\gamma\,:\,Y \,
\longrightarrow\, X$ is a covering as in \eqref{seq2} with Galois group $\Gamma\,=
\,\Gal(\gamma)$. Then the action of $\Gamma$ on $E'$ produces a left action of $\Gamma$
on $\mathbb{P}(E')$. Clearly, we have
$\mathbb{P}(E')/\Gamma \,=\, \mathbb{P}(E_{\ast})$. The isotropy subgroups, for the action
of $\Gamma$ on $\mathbb{P}(E')$, act trivially on the corresponding fibers of
$\mathcal{O}_{\mathbb{P}(E')}\bigl(N(E_{\ast})\bigr)$; this
follows from condition (5) above in the definition of a ramified 
principal ${\rm GL}(n,\mathbb{C})$-bundle and 
\eqref{seq3}. Hence the quotient space
$\mathcal{O}_{\mathbb{P}(E')}\bigl(N(E_{\ast})\bigr)/\Gamma$ is actually a line bundle over
$\mathbb{P}(E')/\Gamma\,=\, \mathbb{P}(E_{\ast})$. We have a natural isomorphism of line
bundles $$\mathcal{O}_{\mathbb{P}(E')}\bigl(N(E_{\ast})\bigr)/\Gamma
\,=\, \mathcal{O}_{\mathbb{P}(E_{\ast})}(1).$$

Note that the pullback of $\mathcal{O}_{\mathbb{P}(E_{\ast})}(1)$ to $\mathbb{P}(E')$ under 
the finite quotient map $$\mathbb{P}(E') \,\longrightarrow \,\mathbb{P}(E')/\Gamma\, =\,
\mathbb{P}(E_{\ast})$$ is $\mathcal{O}_{\mathbb{P}(E')}\bigl(N(E_{\ast})\bigr)$.
The parabolic bundle $E_{\ast}$ is parabolic
ample (respectively, parabolic nef) if and only if the line bundle $\mathcal{O}_{\mathbb{P}(E_{\ast})}(1)$
on $\mathbb{P}(E_{\ast})$ is ample (respectively, nef) \cite[Proposition 3.3]{BL11}.
Also, the following three statements are equivalent:
\begin{itemize}
\item the parabolic vector bundle $E_{\ast}$ is nef,

\item $E'$ is nef, and

\item $\mathcal{O}_{\mathbb{P}(E')}\bigl(N(E_*)\bigr)$ is nef
\end{itemize}
(see \cite[Proposition 3.2]{BS99}).

\section{Definition of Seshadri constant of parabolic bundles}\label{defn-sc}

Let $E_{\ast}$ be a parabolic nef vector bundle of rank $n$ on a smooth projective variety $X$. We fix a
point $x\,\in \,X$, and let $$\psi_x \,\,:\,\, \Bl_x(X) \,\longrightarrow\, X$$
be the blow up of $X$ at $x$ with exceptional divisor $A_x \,=\, \psi^{-1}_x(x)$. Consider
the following fiber product diagram:
\begin{equation}\label{esc}
\begin{tikzcd} 
\Bl_{\rho^{-1}(x)} \bigl(\mathbb{P}(E_{\ast})\bigr) \,=\, \mathbb{P}(E_{\ast})\times_X \Bl_x(X)\arrow[r, "\widetilde{\psi_x}"] \arrow[d, "\widetilde{\rho}"]
 & \mathbb{P}(E_{\ast}) \arrow[d,"\rho"]\\
 \Bl_x(X) \arrow[r, "\psi_x" ]
 & X
\end{tikzcd}
\end{equation}
The above morphism $\rho$ is the projectivization of the parabolic bundle $E_{\ast}$ as
constructed in Section \ref{proj}. Hence the map $\rho$ is flat and the fiber of $\rho$ is a $\mathbb{P}^{n-1}$ (see \cite{BL11}).

Let $\xi$ be the numerical equivalence class of the line bundle 
$\mathcal{O}_{\mathbb{P}(E_{\ast})}(1)\, \longrightarrow\, \mathbb{P}(E_{\ast})$.
 
\begin{defi}\label{main-defn}
Let $E_{\ast}$ be a parabolic nef vector bundle on a smooth projective variety $X$.
The \textit{parabolic 
Seshadri constant} of $E_{\ast}$ at a point $x\,\in\, X$, denoted by
$\varepsilon_{\ast}(E_{\ast},\,x)$, is defined to be 
\begin{align*}
\varepsilon_{\ast}(E_{\ast},\,x)\,\,:=\,\,\sup\,\Bigl\{\lambda\,\in\,\mathbb{R}_{>0}\,\, \big\vert
\,\,\widetilde{\psi_x}^{\ast}(\xi)- \lambda\widetilde{\rho}^{\ast}A_x\ \
{\rm is\ nef} \Bigr\}
\end{align*}
(see \eqref{esc}).
\end{defi}

\begin{xrem}
As mentioned before, there is a canonical bijection between the parabolic vector bundles
of rank $n$ with parabolic divisor $D$ and the ramified principal
${\rm GL}(n,\mathbb{C})$-bundles ramified over $D$. In particular, the constructions of
the projectivization $\mathbb{P}(E_{\ast})$ and the
tautological bundle $\mathcal{O}_{\mathbb{P}(E_{\ast})}(1)$
for a parabolic bundle $E_{\ast}$ are uniquely determined by
$E_{\ast}$. So the parabolic Seshadri constant $\varepsilon_{\ast}(E_{\ast},\,x)$
is well-defined. 
\end{xrem} 
 
The next theorem shows that the parabolic Seshadri constant 
$\varepsilon_{\ast}(E_{\ast},\,x)$ can be computed by examining the intersection of $\xi$ 
with certain curves on $\mathbb{P}(E_{\ast})$.

\begin{thm}\label{corl2.1}
Let $E_{\ast}$ be a parabolic nef vector bundle on a smooth complex projective variety 
$X$, and let $x\,\in\, X$ be a point of $X$. Let $\mathcal{C}_{\rho,x}$ be the set of all integral 
curves $C\,\subset\, \mathbb{P}(E_{\ast})$ that intersect the fiber $\rho^{-1}(x)$
while not being contained in $\rho^{-1}(x)$. Then $$\varepsilon_{\ast}(E_{\ast},\,x) 
\,= \,\inf\limits_{C\in \mathcal{C}_{\rho,x}} \left\{ \frac{\xi\cdot 
C}{\mult_x\rho_{\ast}C}\right\}.$$
 \end{thm}
 
\begin{proof}
First note that $C \,\in \,\mathcal{C}_{\rho,x}$ if and only if $\mult_x\rho_{\ast}C 
\,>\, 0$. Let $\widetilde{C}$ be the strict transform of a curve $C \,\in\,
\mathcal{C}_{\rho,x}$ under the map $\widetilde{\psi_x}$ in \eqref{esc}.
We have $$ \widetilde{\rho}^{\ast}A_x\cdot \widetilde{C} \,=\, A_x\cdot 
\widetilde{\rho}_{\ast}\widetilde{C} \,=\, \mult_x\rho_{\ast}C.$$
Hence
$$\bigl\{\widetilde{\psi_x}^{\ast}(\xi)- \lambda\widetilde{\rho}^{\ast}A_x\bigr\}
\cdot \widetilde{C} \,\geq\, 0$$
if and only if $\xi\cdot C\,\geq\, \lambda \mult_x\rho_{\ast}C$.

Let $C'$ be an irreducible curve in $\Bl_{\rho^{-1}(x)}\bigl(\mathbb{P}(E_{\ast})\bigr)$
which is contained in the exceptional locus $\widetilde{\rho}^{-1}(A_x)$ of $\widetilde{\psi_x}$. Then 
$$\widetilde{\rho}^{\ast}A_x\cdot C'\,=\, A_x \cdot \widetilde{\rho}_{\ast}C' \,<\, 0.$$

Since $E_{\ast}$ is given to be parabolic nef, the line bundle $\xi$ is nef.
Therefore, we conclude that
$$\bigl\{\widetilde{\psi_x}^{\ast}(\xi)- \lambda\widetilde{\rho}^{\ast}A_x\bigr\} \cdot C'
\,=\, \bigl\{ \xi \cdot \widetilde{{\psi}_x}_{\ast}\widetilde{C'} - \lambda A_x \cdot
\widetilde{\rho}_{\ast}C'\bigr\}\geq 0$$ for all $\lambda \,\geq\, 0$.
Every irreducible curve on $\Bl_{\rho^{-1}(x)}\bigl(\mathbb{P}(E_{\ast})\bigr)$
satisfies one of the following conditions:
\begin{itemize}
\item It is the strict transform of a curve $C \,\in\, \mathcal{C}_{\rho,x}$.
 
\item It is contained in the exceptional locus $\widetilde{\rho}^{-1}(A_x)$ of
$\widetilde{\psi_x}$.

\item It does not intersect $\widetilde{\rho}^{-1}(A_x)$. 
\end{itemize}

Now using the nefness of $\widetilde{\psi_x}^{\ast}(\xi)$, we obtain the result.
\end{proof}
 
\begin{xrem}
Let $\rho\,:\, \mathbb{P}(E_{\ast})\,\longrightarrow\, X$ be the
natural projection map, and let $\xi 
\,=\, \mathcal{O}_{\mathbb{P}(E_{\ast})}(1)$, which is constructed in Section \ref{proj}. For 
any point $x\,\in\, X$, the fiber $\rho^{-1}(x)$ is isomorphic to the projective space 
$\mathbb{P}^{n-1}$, and $\xi\big\vert_{\rho^{-1}(x)} \,=\, 
\mathcal{O}_{\mathbb{P}^{n-1}}(N(E_{\ast}))$.

Let $y\,\in\, \mathbb{P}(E_{\ast})$. We can relate the classical Seshadri constant 
$\varepsilon(\xi,\,y)$ of $\xi$ and the parabolic Seshadri constant
$\varepsilon_{\ast}(E_{\ast},\,\rho(y))$ of $E_{\ast}$ at $\rho(y)$ as follows:
$$\varepsilon(\xi,\,y) \,=\,\ \, \inf\limits_{y\in C \subset \mathbb{P}(E_{\ast})} \left\{\frac{\xi\cdot C}{\mult_yC}\right\}
$$
$$
= \,\min\left\{\inf\limits_{C\not\subset \rho^{-1}(\rho(y)) \subseteq \mathbb{P}(E_{\ast})}
\Bigl\{ \frac{\xi\cdot C}{\mult_yC}\Bigr\},\ \inf\limits_{C\subset \rho^{-1}(\rho(y)) \subseteq \mathbb{P}(E_{\ast})} \Bigl\{\frac{\xi\cdot C}{\mult_yC}\Bigr\}\right\}.$$
Hence we have
$$
\varepsilon(\xi,\,y)\,\geq\, \min\left\{\inf\limits_{C\not\subset \rho^{-1}(\rho(y))
\subseteq\mathbb{P}(E_{\ast})} \Bigl\{ \frac{\xi\cdot C}{\mult_yC}\Bigr\},\ N(E_{\ast})\right\}
\,\geq\, \min\Bigl\{\varepsilon_{\ast}(E_{\ast},\,\rho(y)),\ N(E_{\ast})\Bigr\}.$$
\end{xrem}

Now we give another alternative characterization of parabolic Seshadri constants. This
characterization uses 
the Galois covering $\gamma\,:\, Y \,\longrightarrow\, X$ and the orbifold bundle $E'$ over $Y$ associated to 
$E_{\ast}$; see Section \ref{orbifold}.

\begin{thm}\label{thm2.2}
Let $E_{\ast}$ be a parabolic nef vector bundle on a smooth irreducible complex
projective variety $X$, and let $E'\,\longrightarrow\, Y$ be the corresponding orbifold
bundle over $Y$, where $\gamma\,:\, Y\,\longrightarrow\, X$ is a covering as in \eqref{seq2}
with Galois group $\Gamma \,=\, \Gal(\gamma)$. Let
$$\phi_x \,:\, \Bl_{\gamma^{-1}(x)}Y \,\longrightarrow\, Y$$ be the blow up of $Y$
along $\gamma^{-1}(x)$. Consider the following fiber product diagram: 
 \begin{center}
\begin{tikzcd} 
\mathbb{P}(\phi_x^{\ast}E')\, =\, \mathbb{P}(E')\times_Y \Bl_{\gamma^{-1}(x)}Y\arrow[r, "\widetilde{\phi_x}"] \arrow[d, "\widetilde{\tau}"]
 & \mathbb{P}(E') \arrow[d,"\tau"]\\
 \Bl_{\gamma^{-1}(x)}Y \arrow[r, "\phi_x" ]
 & Y
\end{tikzcd}
\end{center}
Denote the numerical equivalence class of the line bundle $\mathcal{O}_{\mathbb{P}(E')}(1)$ by $\xi'$. Let $A_{\gamma^{-1}(x)}$ be the exceptional divisor of the map $\phi_x$. Then
\begin{eqnarray*}
\varepsilon_{\ast}(E_{\ast},\,x)\, &=& \, N(E_{\ast})\cdot \sup\Bigl\{ \lambda
\,\in\, \mathbb{R}_{>0}\,\,\big\vert\,\, \widetilde{\phi_x}^{\ast}(\xi')-
\lambda\widetilde{\tau}^{*}(A_{\gamma^{-1}(x)})\, \rm{~is ~nef} \Bigr\} \\
&=&\, \it N(E_{\ast})\cdot \inf\limits_{C\in \mathcal{C}_{\tau,\gamma^{-1}(x)}} \left\{ \frac{\xi'\cdot C}{\sum\limits_{y\in \gamma^{-1}(x)}\mult_y\tau_{\ast}C}\right\},
\end{eqnarray*}
where $\mathcal{C}_{\tau,\gamma^{-1}(x)}$ is the set of all irreducible curves $C$ in
$\mathbb{P}(E')$ such that $$C \,\nsubseteq \,\tau^{-1}(\gamma^{-1}(x))$$ and
$\sum\limits_{y\in \gamma^{-1}(x)}\mult_y\tau_{\ast}C \,>\, 0$, while
$N(E_{\ast})$ is defined as in \eqref{seq3}.
\end{thm}

\begin{proof}
Note that $\mathbb{P}(E')\times_Y \Bl_{\gamma^{-1}(x)}Y = \Bl_{\tau^{-1}(\gamma^{-1}(x))}\mathbb{P}(E')$. Consider the following commutative diagram:
\begin{center}
\begin{tikzcd} 
\Bl_{\gamma^{-1}(x)} Y\arrow[r, "\widetilde{\gamma}"] \arrow[d, "\phi_x"]
 & \Bl_xX \arrow[d,"\psi_x"]\\
 Y \arrow[r, "\gamma" ]
 & X
\end{tikzcd}
\end{center}
such that $\widetilde{\gamma}^{\ast}(A_x) \,= \,A_{\gamma^{-1}(x)}$. Also, we have the
following commutative diagram:

\begin{center}
\begin{tikzcd} 
\mathbb{P}(E')\times_Y \Bl_{\gamma^{-1}(x)}Y\arrow[r, "\widetilde{\phi_x}"] \arrow[d, "\widetilde{\tau}"] 
& \mathbb{P}(E') \arrow[r, "\gamma'"] \arrow[d,"\tau"] 
& \mathbb{P}(E_{\ast}) = \mathbb{P}(E')/\Gamma \arrow[d,"\rho"] \\
\Bl_{\gamma^{-1}(x)}Y \arrow[r, "\phi_x" ] & Y \arrow[r, "\gamma"]
& X = Y/\Gamma.
\end{tikzcd}
\end{center}
The universal property of a fiber product furnishes the following commutative diagram: 
\begin{center}
\begin{tikzcd}
\mathbb{P}(E')\times_Y \Bl_{\gamma^{-1}(x)}Y 
\arrow[drr, bend left, "\gamma'\circ \widetilde{\phi_x}"]
\arrow[ddr, bend right, "\widetilde{\gamma}\circ \widetilde{\tau}"]
\arrow[dr, "\Delta"] & & \\
& \mathbb{P}(E_{\ast}) \times_{X} \Bl_xX \arrow[r, "\widetilde{\psi}_x"] \arrow[d, "\widetilde{\rho}"]
& \mathbb{P}(E_{\ast}) \arrow[d, "\rho"] \\
& \Bl_xX \arrow[r, "\psi_x"]
& X
\end{tikzcd}
\end{center}
Recall that $\gamma'^{\ast}(\xi)\,=\, N(E_{\ast})\xi'$. Hence for any $\lambda \,> \,0$,
the following three statements are equivalent:
\begin{itemize}
\item the line bundle $\widetilde{\psi_x}^{\ast}(\xi)- \lambda\widetilde{\rho}^{\ast}A_x$
is nef,

\item $\Delta^{\ast}\bigl(\widetilde{\psi_x}^{\ast}(\xi)-
\lambda\widetilde{\rho}^{\ast}A_x\bigr)$ is nef, and

\item $N(E_{\ast})\widetilde{\phi_x}^{\ast}(\xi')-
\lambda\widetilde{\tau}^{*}(A_{\gamma^{-1}(x)})$ is nef.
\end{itemize}
Therefore, we conclude that
\begin{align*}
\varepsilon_{\ast}(E_{\ast},\,x) \,=\, N(E_{\ast})\cdot\sup\Bigl\{ \lambda\,\in\, \mathbb{R}_{>0}
\,\,\big\vert\,\, \widetilde{\phi_x}^{\ast}(\xi')- \lambda\widetilde{\tau}^{*}(A_{\gamma^{-1}(x)})
\,\ \rm{is\,~nef} \Bigr\}.
\end{align*}

Next we claim that 
\begin{equation}\label{ec}
\sup\Bigl\{ \lambda \,\in\, \mathbb{R}_{>0}\,\, \big\vert\,\, \widetilde{\phi_x}^{\ast}(\xi')-
\lambda\widetilde{\tau}^{*}(A_{\gamma^{-1}(x)})\,\ {\rm is\,~nef}\Bigr\}\,= \,\,
\inf\limits_{C\in \mathcal{C}_{\tau,\gamma^{-1}(x)}}
\left\{ \frac{\xi'\cdot C}{\sum\limits_{y\in \gamma^{-1}(x)}\mult_y\tau_{\ast}C}\right\}.
\end{equation}

To prove \eqref{ec}, let $\widetilde{C}$ be the strict transform of a curve $C \,\in \,
\mathcal{C}_{\tau,\gamma^{-1}(x)}$ under the blow up map
$$\widetilde{\phi_x} \,:\, \Bl_{\tau^{-1}(\gamma^{-1}(x))} \mathbb{P}(E')
\,\longrightarrow\, \mathbb{P}(E').$$ Note that $$ \widetilde{\tau}^{\ast}A_{\gamma^{-1}(x)}
\cdot \widetilde{C}\,=\, A_{\gamma^{-1}(x)}\cdot \widetilde{\tau}_{\ast}\widetilde{C}
\,=\, \sum\limits_{y\in \gamma^{-1}(x)}\mult_y\tau_{\ast}C.$$
Hence $\bigl\{\widetilde{\phi_x}^{\ast}(\xi')- \lambda\widetilde{\tau}^{\ast}
A_{\gamma^{-1}(x)}\bigr\} \cdot \widetilde{C} \,\geq\, 0$ if and only if
$$\xi'\cdot C \geq \,\lambda\, \sum\limits_{y\in \gamma^{-1}(x)}\mult_y\tau_{\ast}C .$$
If $C'$ is an irreducible curve in $\Bl_{\tau^{-1}(\gamma^{-1}(x))}\mathbb{P}(E')$,
and it is contained in the exceptional locus $\widetilde{\tau}^{-1}(A_{\gamma^{-1}(x)})$ of
$\widetilde{\phi_x}$, then 
$$\bigl\{\widetilde{\phi_x}^{\ast}(\xi')- \lambda\widetilde{\tau}^{\ast}
A_{\gamma^{-1}(x)}\bigr\} \cdot \widetilde{C} \,= \,
\bigl\{ \xi' \cdot \widetilde{{\phi}_x}_{\ast}\widetilde{C} -
\lambda A_{\gamma^{-1}(x)} \cdot \widetilde{\tau}_{\ast}C'\bigr\}\,\geq\, 0$$
for all $\lambda \,\geq\, 0$.

Any irreducible curve on $\Bl_{\tau^{-1}(\gamma^{-1}(x))} \mathbb{P}(E')$
satisfies one of the following three conditions:
\begin{itemize}
\item It is the strict transform of a curve $C \,\in\, \mathcal{C}_{\tau,\gamma^{-1}(x)}$.

\item It is contained in the exceptional locus $\widetilde{\tau}^{-1}(A_{\gamma^{-1}(x)})$
of $\widetilde{\phi_x}$.

\item It does not intersect $\widetilde{\tau}^{-1}(A_{\gamma^{-1}(x)})$. 
\end{itemize}
Now using the nefness of $\widetilde{\phi_x}^{\ast}(\xi')$ we conclude that \eqref{ec} holds. 
This also completes the proof of the theorem.
\end{proof}
We use the notation of Theorem \ref{thm2.2} in what follows.
\begin{thm}[{Seshadri criterion for parabolic ampleness}]\label{seshadri-crit}
Let $E_{\ast}$ be a parabolic vector bundle on a smooth irreducible projective variety 
$X$ such that the numerical class $\xi' \,\equiv\, \mathcal{O}_{\mathbb{P}(E')}\bigl(N(E_*)\bigr)$ is $\tau$-ample. Then
$E_{\ast}$ is parabolic ample if and only if $$\inf\limits_{x\in 
X}\varepsilon_{\ast}(E_{\ast},\,x)\,>\, 0,$$ where the infimum is taken over all points of $X$.
\end{thm}

\begin{proof}
Suppose that $E_{\ast}$ is parabolic ample. Then the numerical equivalence class of the 
tautological line bundle $\xi\,=\,[\mathcal{O}_{\mathbb{P}(E_{\ast})}(1)]$ is ample on 
$\mathbb{P}(E_{\ast})$. Let $h$ be an ample divisor class on $X$ such that 
$\xi-\delta\rho^{\ast}h$ is ample for a sufficiently small $\delta\,>\,0$. Now for all 
$C\,\in\,\mathcal{C}_{\rho,x}$, we have $$(\xi-\delta\rho^{\ast}h)\cdot C\,>\,0.$$ Then
\begin{align*}
\frac{\xi\cdot C}{ \text{mult}_x\rho_{\ast}C}\,=\,
\frac{(\xi-\delta\rho^{\ast}h)\cdot C}{\text{mult}_x\rho_{\ast}C}+
\frac{\delta\rho^{\ast}h\cdot C}{\text{mult}_x\rho_{\ast}C}\,\geq\,
\frac{\delta\rho^{\ast}h\cdot C}{\text{mult}_x\rho_{\ast}C}\,\geq\, 
\delta \frac{ h \cdot\rho_{\ast} C}{\text{mult}_x\rho_{\ast}C}\,\geq\, 
\delta \varepsilon(h,\,x)\,>\,0. 
\end{align*}
The last inequality follows using the usual Seshadri criterion for ample divisors,
because $h$ is ample; see \cite[Theorem 1.4.13]{Laz}. Hence
\begin{align*}
\inf\limits_{C\in\mathcal{C}_{\rho,x}}\Big\{\frac{\xi\cdot C}{ \text{mult}_x\rho_{\ast}C}\Big\}
\,\geq\, \delta \varepsilon(h,\,x)\,>\,0.
\end{align*}
Consequently, from Theorem \ref{corl2.1} it is deduced
that $$\varepsilon_{\ast}(E_{\ast},\,x\,)\geq\, \delta \varepsilon(h,\,x)\, >\, 0$$
for every point $x\,\in\, X$. Thus we have
$$\inf\limits_{x\in X}\varepsilon_{\ast}(E_{\ast},\,x) \,\ge\,
\delta \inf\limits_{x\in X}\varepsilon(h,\,x) \,>\, 0.$$
The last inequality again follows from
the usual Seshadri criterion applied to the ample divisor $h$. 

To prove the converse,
assume that $\inf\limits_{x\in X}\varepsilon_{\ast}(E_{\ast},\,x) \,> \,0$
for all $x\, \in\, X$. Therefore, using Theorem \ref{thm2.2} we have
\begin{align}\label{seq4}
\inf\limits_{x\in X}\,\,\inf\limits_{C\in \mathcal{C}_{\tau,\gamma^{-1}(x)}}
\left\{\frac{\xi'\cdot C}{\sum\limits_{y\in \gamma^{-1}(x)}\mult_y\tau_{\ast}C}\right\}\,>\, 0.
\end{align}
Assume that $E_*$ is not parabolic ample, i.e., $\xi'\,\equiv\, \mathcal{O}_{\mathbb{P}(E')}(N(E_*))$ is not ample. Note that as
$\xi'$ is $\tau$-ample, we have $\xi'$ is nef, i.e., $E_*$ is parabolic nef. Then by the Seshadri criterion for ample divisors, one
has $$\inf\limits_{z\in \mathbb{P}(E')}\varepsilon(\xi',z) \,=\, 0.$$
So for each $m\,\in\, \mathbb{N}$, there exist points $z_m \,\in\, \mathbb{P}(E')$ and irreducible curves $C_m\,\subset\,
\mathbb{P}(E')$ through $z_m$ such that $$\xi'\cdot C_m \,<\, \frac{1}{m}\mult_{z_m}C_m.$$
We claim that the curve $C_m$ is not contracted by $\tau$ for infinitely many $m.$

If the claim is not true, $C_m$ are contracted by $\tau$ for all $m\,\gg\, 0$.
Then for all $m\,\gg\, 0$
\begin{align*}
\inf\limits_{y\in Y}\inf\limits_{z\in \tau^{-1}(y)}\varepsilon\bigl(\xi'\vert_{\tau^{-1}(y)},z\bigr)
\,\leq\, \frac{\xi'\cdot C_m}{\mult_{z_m}C_m} \,<\, \frac{1}{m}, 
\end{align*}
so that 
\begin{align}\label{seq6}
 \inf\limits_{y\in Y}\inf\limits_{z\in \tau^{-1}(y)}\varepsilon\bigl(\xi'\vert_{\tau^{-1}(y)},z\bigr) \,=\, 0.
\end{align}

Choose an ample divisor $h$ on $Y$ so that $\xi'+\tau^*h$ is ample. Hence by the Seshadri ampleness criterion for ample
divisors, we conclude that
\begin{align*}
 \inf\limits_{y\in Y}\inf\limits_{z\in \tau^{-1}(y)}\varepsilon\bigl((\xi'+\tau^*h)\vert_{\tau^{-1}(y)},z\bigr) \,=\,
\inf\limits_{y\in Y}\inf\limits_{z\in \tau^{-1}(y)}\varepsilon\bigl(\xi'\vert_{\tau^{-1}(y)},z\bigr) \,>\, 0,
\end{align*}
which in fact contradicts \eqref{seq6}. This proves our claim. 

Therefore 
$\tau\vert_{C_{m_k}} \,:\, C_{m_k} \,\longrightarrow \,Y$ is finite for 
an infinite sequence $\{m_k\}_{k=1}^{\infty}$ of distinct positive integers. 
Let $y_{m_k} = \tau(z_{m_k})$. Then by \cite[Lemma 2.3]{F21}, we have $$\mult_{y_{m_k}}\tau_*{C_{m_k}} \geq \mult_{z_{m_k}}C_{m_k}.$$ 
This contradicts \eqref{seq4}. Thus $E_{\ast}$ is parabolic ample.
\end{proof}

The next theorem gives an upper bound for the parabolic Seshadri constants.

\begin{thm}\label{bound}
Let $E_*$ be a parabolic nef vector bundle on $X$, and $x\,\in\, X$ a point. Then 
$$\varepsilon_*(E_*,\,x) \,\leq\, \Bigl(\frac{N(E_*)^{\dim \tau(W)}\xi'^{\dim W}\cdot
[W]}{\binom{\dim W}{\dim\tau(W)} \cdot \lvert \Gamma \rvert(\xi'^{\dim W_{\gamma^{-1}(x)}}
[W_{\gamma^{-1}(x)}])}\Bigr)^{\frac{1}{\dim\tau(W)}},$$
as $W$ ranges through the subvarieties of $\mathbb{P}(E')$ that meet
$\tau^{-1}(\gamma^{-1}(x))$ without being contained in $\tau^{-1}(\gamma^{-1}(x))$, where
$\tau$ and $\gamma$ are as in Theorem $\ref{thm2.2}$. 
In the above inequality, $W_{\gamma^{-1}(x)} := \tau^{-1}(\gamma^{-1}(x))\bigcap W$.
\end{thm}

\begin{proof}
Consider the following commutative diagram
\begin{center}
\begin{tikzcd} 
\mathbb{P}(\phi_x^{\ast}E') \,=\, \mathbb{P}(E')\times_Y \Bl_{\gamma^{-1}(x)}Y \arrow[r, "\widetilde{\phi_x}"] \arrow[d, "\widetilde{\tau}"]
& \mathbb{P}(E') \arrow[d,"\tau"]\\
\Bl_{\gamma^{-1}(x)}Y \arrow[r, "\phi_x" ]
& Y
\end{tikzcd}
\end{center}
where $\tau$ is the projectivization of the orbifold bundle $E'$ over $Y$, and $\phi_x$ is 
the blow up of $Y$ at $\gamma^{-1}(x)$. Recall from Theorem \ref{thm2.2} that 
$$\varepsilon_{\ast}(E_{\ast},\,x)\,=\, N(E_{\ast})\cdot
\sup\Bigl\{ \lambda \in \mathbb{R}_{>0} \,\,\big\vert\,\, 
\widetilde{\phi_x}^{\ast}(\xi')- \lambda\widetilde{\tau}^{*}(A_{\gamma^{-1}(x)}) 
\ \, \text{is\, nef}\Bigr\}.$$ Let $W\,\subseteq\, \mathbb{P}(E')$ be a 
subvariety that meets $\tau^{-1}(\gamma^{-1}(x))$ without being contained in 
$\tau^{-1}(\gamma^{-1}(x))$. Let $W' \,\subset\, \mathbb{P}(\phi_x^{\ast}E')$ be the strict 
transform of $W$ by the blow up morphism $\widetilde{\phi}_x$.
Then by the above observation we have 
$$\bigl(N(E_{\ast})\widetilde{\phi}_x^*(\xi')-\varepsilon_*(E_*,\,x)\widetilde{\tau}^*(A_{\gamma^{-1}(x)})\bigr)^{\dim 
W'}\cdot W' \,\geq\, 0.$$
We now specialize to the situation where 
$$W'\,=\,\mathbb{P}(\phi_x^{\ast}E').$$

Let $\dim Y\,=\,n$ and $\dim \mathbb{P}(E')\,=\, n+e$. Thus
$$\Bigl(N(E_{\ast})\widetilde{\phi_x}^*(\xi')-
\varepsilon_*(E_*,\,x)\widetilde{\tau}^*\bigl(A_{\gamma^{-1}(x)}\bigr)\Bigr)^{n+e}\,\geq\, 0.$$
Note that the self-intersection number $A_{\gamma^{-1}(x)}^k\,=\,0$ for $k\,>\,n$. Expanding binomially we thus get
$$ \sum\limits_{k=0}^n\binom{n+e}{k}\bigl(-\varepsilon_*(E_*,\,x)
\widetilde{\tau}^*A_{\gamma^{-1}(x)}\bigr)^k N(E_{\ast})^{n+e-k}
\widetilde{\phi_x}^*(\xi'^{n+e-k}) \,\geq\, 0.$$
Note that $\widetilde{\phi}_{x_{\ast}}\bigl(-(-\widetilde{\tau}^*A_{\gamma^{-1}(x)})^k\bigr)$
are pseudoeffective since the restriction $-A_{\gamma^{-1}(x)}\big\vert_{A_{\gamma^{-1}(x)}}$
is ample. Hence using the projection formula we conclude that
\begin{equation}\label{12}
N(E_{\ast})^{n+e} \xi'^{n+e} + \binom{n+e}{n}N(E_{\ast})^e
\bigl[-\varepsilon_*(E_*,\,x)\widetilde{\tau}^*A_{\gamma^{-1}(x)}\bigr]^n
\widetilde{\phi_x}^*(\xi'^e) \,\geq\, 0.
\end{equation}

Now we have the following equalities: 
\begin{eqnarray*}
\bigl[-\varepsilon_*(E_*,\,x)\widetilde{\tau}^*A_{\gamma^{-1}(x)}\bigr]^n\widetilde{\phi_x}^*(\xi'^e) &=& \varepsilon_*^n(E_*,\,x)\widetilde{\phi}_{x_{\ast}}\bigl((-\widetilde{\tau}^*A_{\gamma^{-1}(x)})^n\bigr)\xi'^{e}\\
&=& \varepsilon_*^n(E_*,\,x)\tau^*\bigl[ (\phi_{x_{\ast}}(-A_{\gamma^{-1}(x)}))^n\bigr] \xi'^e\\ 
&=& \varepsilon_*^n(E_*,\,x)\tau^*(-\gamma^{-1}(x))\xi'^e\\
&=& -\varepsilon_*^n(E_*,\,x)\lvert \Gamma\rvert\bigl(\xi'^e\cdot [W_{{\gamma}^{-1}(x)}]\bigr), 
\end{eqnarray*}
where $W_{\gamma^{-1}(x)}\,=\,\tau^{-1}(\gamma^{-1}(x))\bigcap W$. 

{}From \eqref{12} it follows that
$$N(E_{\ast})^{n}\cdot\xi'^{n+e}-\binom{n+e}{n}\varepsilon_*^n(E_*,\,x)\lvert
\Gamma\rvert\bigl(\xi'^e\cdot [W_{{\gamma}^{-1}(x)}]\bigr) \,\geq\, 0.$$
Rearranging the terms we obtain the required bound on $\varepsilon_*(E_*,\,x)$.
\end{proof}

For a parabolic vector bundle $E_{\ast}$, we define $\mu^{par}_{\min}(E_{\ast})$ to be the 
parabolic slope of the minimal parabolic semistable subquotient of $E_{\ast}$, or in other 
words, $\mu^{par}_{\min}(E_{\ast})$ is the parabolic slope of the final piece of the graded 
object for the Harder-Narasimhan filtration of $E_{\ast}$. Note that if $E'$ is the 
orbifold bundle on $Y$ corresponding to $E_{\ast}$ for the Galois morphism $\gamma \,:\, Y 
\,\longrightarrow\, X$ as in (\ref{seq2}), then we have $$\mu_{\min}(E') \,=\,
\lvert \Gal(\gamma)\rvert \cdot \mu^{par}_{\min}(E_{\ast}).$$

\begin{thm}
Let $E_{\ast}$ be a parabolic ample vector bundle over
a smooth irreducible projective curve $C$ with parabolic divisor $D$. Then for any
point $x\,\in\, C$, the parabolic Seshadri constant satisfies the following: 
$$\varepsilon_{\ast}(E_{\ast},\,x)\,=\, N(E_{\ast})\cdot\mu^{par}_{\min}(E_{\ast})\ \text{ when
 \,\,$x\,\notin\, D$,\, and}$$
$$\varepsilon_{\ast}(E_{\ast},\,x)\,\geq\, N(E_{\ast})\cdot \mu^{par}_{\min}(E_{\ast})
\ \text{ when\, $x\,\in\, D$.}$$
In particular, $\varepsilon_{\ast}(E_{\ast},\,x) \,\geq\,
\frac{N(E_{\ast})}{\rank(E)}$ for every point $x\,\in\, C$.
\end{thm}

\begin{proof}
As in \eqref{seq2}, let $\gamma\,:\,C'\,\longrightarrow\, C$ be a
ramified Galois covering with Galois group 
$\Gal(\gamma)\, = \,\Gamma$. Then the cardinality of each fiber is
$$\lvert\gamma^{-1}(x)\rvert\,=
\,\lvert\Gamma\rvert$$ if $x\,\notin\, D$, and it is
$\lvert\gamma^{-1}(x)\rvert\,\leq\,\lvert\Gamma\rvert$ if $x\,\in\, D$.

Since $C$ is a curve, the blow up map $\phi_x$ in Theorem \ref{thm2.2} is the identity
morphism, and hence $\widetilde{\phi_x}$ is also the identity morphism. Consequently,
\begin{align*}
\varepsilon_{\ast}(E_{\ast},\,x)\,\,=\,\, N(E_{\ast})\sup\Bigl\{ \lambda\,\in\, \mathbb{R}_{>0}
\,\,\big\vert\,\, \xi'- \lambda\tau^{-1}(\gamma^{-1}(x))\ \text{ is\, nef}\Bigr\}.
\end{align*}
Denote the numerical class of fiber of $\tau$ by $f$. Then 
\begin{align*}
\varepsilon_{\ast}(E_{\ast},\,x) \,\,=\,\, N(E_{\ast})\cdot\sup\Bigl\{ \lambda\,\in\, \mathbb{R}_{>0}
\,\,\big\vert\,\, \xi'- \lambda\big\vert\gamma^{-1}(x)\big\vert f\ \text{ is\, nef}\Bigr\}.
\end{align*}
Therefore,
\begin{align*}
\varepsilon_{\ast}(E_{\ast},\,x)\, =\, N(E_{\ast})\frac{1}{\big\vert\gamma^{-1}(x)\big\vert}
\varepsilon(E',\,y),
\end{align*}
where $y\,\in \,C'$. From \cite[Theorem 3.1]{Hac00}, it follows that $\varepsilon(E',\,y)
\,=\,\mu_{\text{min}}(E').$ Using \cite[Equation 2.15]{B98}, we also have
$\mu_{\text{min}}(E') \,= \,\big\vert\Gamma\big\vert \mu_{\text{min}}^\text{par}(E_{\ast})$.
Hence the theorem follows.
 \end{proof}

\begin{xrem}
If $E_{\ast}$ is a parabolic nef bundle but it is not ample, then the corresponding orbifold
bundle $E'$ is also nef but not ample. In this case
$$\varepsilon(E',\,y)\,=\,\mu_{\text{min}}(E')\,=\,0.$$
Hence $\varepsilon_{\ast}(E_{\ast},\,x)\,=\,0$.
\end{xrem}

\section{Properties of parabolic Seshadri constants}\label{properties-sc}

In this section we define the notion of \textit{multipoint Seshadri constants} for nef 
vector bundles on irreducible projective varieties, generalizing the case of a single
point. We will then give a description of the parabolic Seshadri constants using
restriction to curves.

\begin{defi}\label{defi3.1}
Let $E$ be a vector bundle on an irreducible complex projective variety $Y$, and
let $Z\,=\,\{z_1,\,\cdots,\,z_n\}$ be a set of distinct points of $Y$. Let
$\phi_z \,:\, \Bl_ZY\,\longrightarrow\, Y$ be the blow-up of $Y$ along $Z$.
Consider the following commutative diagram:
\begin{center}
\begin{tikzcd} 
\mathbb{P}(\phi^{\ast}E) \arrow[r, "\widetilde{\phi_z}"] \arrow[d, "\widetilde{\tau}"]
 & \mathbb{P}(E) \arrow[d,"\tau"]\\
 \Bl_{Z}Y \arrow[r, "\phi_z" ]
 & Y
\end{tikzcd}
\end{center}
Let $\xi$ and $\xi'$ be the tautological bundles on $\mathbb{P}(E)$ and
$\mathbb{P}(\phi^{\ast}E)$ respectively, and let
$$A_{z_i}'\,=\,\widetilde{\phi_z}^{-1}(\tau^{-1})(z_i)$$ for each $1\,\leq\, i \,\leq\, n$.
Then the \textit{multipoint Seshadri constant} of $E$ at $Z$, denoted by $\varepsilon(E,\,Z)$,
is defined as follows:
$$
\varepsilon(E,\,Z)\,\,:=\,\,\sup\bigl\{\lambda\,>\,0\,\,\big\vert\,\, \xi'-\lambda\sum\limits_{i=1}^n
A_{z_i}'\ \text{ is\, nef}\bigr\}.
$$
\end{defi}

It is straight-forward to verify the following equivalent formulation: $$\varepsilon(E,\,Z) 
\,\,=\,\,\inf\limits_{\mathcal{C}_{\tau,Z}}\left\{\frac{\xi \cdot C}{\sum\limits_{i=1}^n 
\mult_{z_i} \tau_{*}C}\right\},$$ where $\mathcal{C}_{\tau,Z}$ is the set of all 
irreducible curves $C\,\subset\, \mathbb{P}(E)$ such that
\begin{itemize}
\item $C\bigcap \tau^{-1}(Z) \,\neq\, \emptyset$, i.e., $\sum\limits_{i=1}^n \mult_{z_i} \tau_{*}C
\,> \,0$, and

\item $C \,\nsubseteq \,\tau^{-1}(Z)$.
\end{itemize}
Note that when $Z$ is a singleton, the above definition recovers the usual single point Seshadri constant of a vector bundle. 

The single point Seshadri constants of vector bundles over smooth curves have a particularly nice description.
For a nef vector bundle $E$ over a smooth curve $Y$, for any point $y\,\in\, Y$, we have 
$$\varepsilon(E,\,y)\,\,=\,\,\inf\limits_{C\in \mathcal{C}_{\tau,y}} \Bigl\{ \frac{\xi\cdot C}{\mult_y \tau_*C}\Bigr\}
\,=\, \mu_{\min}(E),$$
where $\mathcal{C}_{\tau,y}$ is the set of curves $C\,\subset\, \mathbb{P}(E)$ such that $C$
intersects the fiber $\tau^{-1}(y)$ but it is not contained in $\tau^{-1}(y)$.

\begin{xrem}\label{xrem4}
Let $C$ be an integral curve and $Z\,=\,\{z_1,\,\cdots,\,z_n\}\, \subset\, C$
a finite subset. Then $\mult_{z_i}C\,\geq\, 1$ for each $1\,\leq \,i \,\leq \,n$.
Set $\nu \,:\, \widetilde{C} \,\longrightarrow\, C$ to be the normalization.
For a vector bundle $E$ over $C$, consider the following fiber product diagram
 \begin{center}
\begin{tikzcd} 
\Bl_{\eta^{-1}(Z)}\mathbb{P}(E) \,=\,\mathbb{P}(\nu^{\ast}E) \arrow[r, "\widetilde{\nu}"] \arrow[d, "\widetilde{\eta}"]
 & \mathbb{P}(E) \arrow[d,"\eta"]\\
 \widetilde{C} \arrow[r, "\nu" ]
 & C
\end{tikzcd}
\end{center}
where $\mathbb{P}(E)$ is the projective bundle over $C$ associated to $E$ and $\eta$ is
the natural map. Let $\xi$ and $\xi'$ denote the numerical classes of
$\mathcal{O}_{\mathbb{P}(\nu^{\ast}E)}(1)$ and $\mathcal{O}_{\mathbb{P}(E)}(1)$ respectively.
Then $$\widetilde{\nu}^{\ast}(\xi)\,=\,\xi'.$$ Let $B' \,\subset\, \mathbb{P}(\nu^{\ast}E)$
(respectively, $B\,\subset\, \mathbb{P}(E)$) be a curve which is not in any fiber
of $\widetilde{\eta}$ (respectively, $\eta$). Suppose that $B'$ is the strict transform
of $B$ under the blow up map $\widetilde{\nu}$. Then $$B'\cdot \xi'\,=\, B\cdot \xi.$$
Hence for any point $y\,\in\, \widetilde{C}$, we have
\begin{align*}
\varepsilon(E,\,Z)\,=\,\inf\limits_{B\in{\mathcal{C}_{\eta,Z}}}\left\{ \frac{\xi\cdot B}{\sum\limits_{i=1}^n \mult_{z_i} \eta_{*}B}\right\} =\frac{1}{\sum\limits_{i=1}^n \mult_{z_i} C} \left\{\inf\limits_{y\in \nu^{-1}(Z)}\inf\limits_{B'\in \mathcal{C}_{\widetilde{\eta},y}} \Bigl\{ \frac{\xi'\cdot B'}{\mult_y\widetilde{\eta}_*B'}\Bigr\}\right\}\\=\frac{1}{\sum\limits_{i=1}^n \mult_{z_i} C}\Bigl\{\inf\limits_{y\in \nu^{-1}(Z)} \varepsilon(\nu^{\ast}E,\,y)\Bigr\} \,=\,\frac{1}{\sum\limits_{i=1}^n \mult_{z_i} C}\Bigl\{\inf\limits_{y\in \nu^{-1}(Z)} \mu_{\min}(\nu^*E)\Bigr\} 
\,=\,\frac{\mu_{\min}(\nu^{\ast}E)}{\sum\limits_{i=1}^n \mult_{z_i} C}.
\end{align*}
\end{xrem}

The next result describes how parabolic Seshadri constants can be computed using 
restriction to curves.

\begin{thm}\label{thm3.2}
Let $E_{\ast}$ be a parabolic nef vector bundle on a smooth irreducible complex projective
variety $X$, and let $E'\,\longrightarrow\, Y$ be the corresponding orbifold bundle over
$Y$ (see Section \ref{orbifold}). Then 
\begin{align*}
\varepsilon_{\ast}(E_{\ast},\,x) \,\,=\,\, N(E_{\ast})\cdot\inf\limits_{C\subset Y}
\left\{\frac{\mu_{\min}(\nu^{\ast}E')}{\sum\limits_{y\in \gamma^{-1}(x)}\mult_y C}\right\},
\end{align*}
where the infimum is taken over all irreducible curves $C\,\subset\, Y$ such that
$C\bigcap \gamma^{-1}(x) \,\neq \,\emptyset$, and $\nu \,:\, \overline{C}
\,\longrightarrow\, C$ is the normalization map.
\end{thm}

\begin{proof}
Let $C\,\subset\, Y$ be an irreducible curve such that $C\bigcap \gamma^{-1}(x)
\,\neq\, \emptyset$. Consider the following diagram:
\begin{center}
\begin{tikzcd}
\mathbb{P}(E'\big\vert_C) \arrow[r, hook] \arrow[d, "\tau_c"]
 & \mathbb{P}(E') \arrow[d,"\tau"]\\
 C \arrow[r, hook ]
 & Y
\end{tikzcd}
\end{center}
Recall that $\mathcal{C}_{\tau,\gamma^{-1}(x)}$ is the set of all irreducible curves $W$ in
$\mathbb{P}(E')$ such that $W\,\nsubseteq \,\tau^{-1}(\gamma^{-1}(x))$ and
$\sum\limits_{y\in \gamma^{-1}(x)}\mult_y\tau_{\ast}W \,>\, 0$. 

Similarly, we denote by $\mathcal{C}_{\tau_c,\gamma^{-1}(x)}$ the set of all irreducible 
curves $B \,\subset \,\mathbb{P}(E'\big\vert_C)$ such that
$$
B\,\nsubseteq\, \tau_c^{-1}(C 
\cap\gamma^{-1}(x))$$ and $\sum\limits_{y\in \gamma^{-1}(x)}\mult_y{\tau_c}_{\ast}B \,>\, 0$. 
Note that
\begin{equation}\label{ao}
\mathcal{C}_{\tau,\gamma^{-1}(x)} \,=\, \bigcup\limits_{C\subset Y} 
\mathcal{C}_{\tau_c,\gamma^{-1}(x)},
\end{equation}
where the union is taken over all irreducible curves $C$ in $Y$ such that $C\bigcap
\gamma^{-1}(x) \,\neq\, \emptyset$. From Theorem \ref{thm2.2} we have
$$
\varepsilon_{\ast}(E_{\ast},\,x) \,\,=\,\, N(E_{\ast})\cdot\inf\limits_{C\in
\mathcal{C}_{\tau,\gamma^{-1}(x)}} \left\{ \frac{\xi'\cdot C}{\sum\limits_{y\in
\gamma^{-1}(x)}\mult_y\tau_{\ast}C}\right\}.
$$
Using \eqref{ao} we have
$$\varepsilon_{\ast}(E_{\ast},\,x)\,\,=\,\, N(E_{\ast})\cdot
\inf\limits_{C \subset Y} \inf\limits_{B\in \mathcal{C}_{\tau_c,\gamma^{-1}(x)}}
\left\{\frac{\xi'\cdot B}{\sum\limits_{y\in \gamma^{-1}(x)} \mult_y{\tau_c}_{\ast}B}\right\},
$$
where the infimum is taken over all irreducible curves $C$ in $Y$ such that $C\bigcap 
\gamma^{-1}(x)\,\neq\, \emptyset$.

Now observe that for an irreducible curve $C \,\subset\, Y$ with $C\bigcap \gamma^{-1}(x)
\,\neq \,\emptyset$, using Remark \ref{xrem4} it follows that
$$\inf\limits_{B\in\mathcal{C}_{\tau_c,\gamma^{-1}(x)}} \left\{\frac{\xi'\cdot
B}{\sum\limits_{y\in \gamma^{-1}(x)} \mult_y{\tau_c}_{\ast}B}\right\}
\,=\, \frac{\mu_{\min}(\nu^{\ast}E')}{\sum\limits_{y\in \gamma^{-1}(x)}\mult_y C}.$$
This completes the proof. 
\end{proof}

For a parabolic vector bundle $E_{\ast}$ of rank $r$ we define its \textit{parabolic 
discriminant}, denoted by $\triangle_{par}(E_{\ast})$, as follows:
\begin{equation}\label{pd}
\triangle_{par}(E_{\ast}) \,:= \,2rc_2(E_{\ast}) - (r-1)c_1^2(E_{\ast}).
\end{equation}

\begin{xrem}\label{par-disc}
Let $E_{\ast}$ be a parabolic nef vector bundle on $X$. Then, by Theorem \ref{thm3.2}, for 
any point $x\,\in\, X$
\begin{align*}
\varepsilon_{\ast}(E_{\ast},\,x) \,=\, N(E_{\ast})\cdot \inf\limits_{C\subset Y} \left\{\frac{\mu_{\min}(\nu^{\ast}E')}{\sum\limits_{y\in \gamma^{-1}(x)}\mult_y C}\right\},
\end{align*}
where the infimum is taken over all irreducible curves $C\,\subset\, Y$ such that
$C\bigcap \gamma^{-1}(x)\,\neq\, \emptyset$, and $\nu\,:\, \overline{C} \,\longrightarrow\,
C$ is the normalization map. Note that
$$\mu_{\min}(\nu^{\ast}E')\,\leq\, \mu(\nu^{\ast}E') \,=\,(\text{det}(E')\cdot C)/\rank(E').$$ 
Thus
$$\varepsilon_{\ast}(E_{\ast},\,x)\,\leq\, \frac{N(E_{\ast})}{\rank(E_{\ast})}
\inf\limits_{C\subset Y}\left\{\frac{\det(E')\cdot C}{\sum\limits_{y\in
\gamma^{-1}(x)} \mult_yC}\right\} \,=\, \frac{N(E_{\ast})}{\rank(E_{\ast})}
\varepsilon\bigl(\det(E'),\,\gamma^{-1}(x)\bigr).$$
Here $\varepsilon\bigl(\det(E'),\,\gamma^{-1}(x)\bigr)$ denotes the multipoint Seshadri 
constants of $E'$ at $\gamma^{-1}(x)\, \subset\, Y$. Note that the rank of the parabolic vector 
bundle $E_{\ast}$ is simply the rank of the underlying vector bundle $E$.
 
Moreover, if $E_{\ast}$ is a semistable parabolic ample vector bundle with 
$\triangle_{par}(E_{\ast})\, =\, 0$ (see \eqref{pd}), then $E'$ is also semistable ample 
bundle on $Y$ with vanishing discriminant. Hence in this case $$\mu_{\min}(\nu^{\ast}E')\,=\,
\mu(\nu^{\ast}E')$$ (see \cite[Theorem 1.2]{BB08}). Therefore, in this case
we get the following equality:
$$\varepsilon_{\ast}(E_{\ast},\,x) \,=\, \frac{N(E_{\ast})}{\rank(E_{\ast})}
\varepsilon\bigl(\det(E'),\,\gamma^{-1}(x)\bigr)$$
for every point $x\,\in\, X$.
\end{xrem}

We now describe the parabolic Seshadri constants of symmetric powers and tensor products of 
parabolic vector bundles.

\begin{thm}\label{thm-sym}
Let $E_{\ast}$ be a parabolic nef vector bundle on a smooth irreducible complex
projective variety $X$. Then for any positive integer $m$
$$\varepsilon_{\ast}(S^m(E_{\ast}),\,x)\,=\, m\varepsilon_{\ast}(E_{\ast},\,x)$$ for every
point $x\,\in\, X$.
\end{thm}

\begin{proof}
Let $E' \,\longrightarrow\, Y$ be the corresponding orbifold bundle over $Y$ as in Section 
\ref{orbifold}. For any positive integer $m$, the parabolic bundle corresponding to the orbifold bundle
$S^m(E')$ is the parabolic symmetric product $S^m(E_{\ast})$ (see \cite{B97I}).
Therefore, from Proposition \ref{prop-na} it follows that $S^m(E_{\ast})$ is nef if
$E_{\ast}$ is nef. Thus by Theorem \ref{thm3.2}, 
\begin{align*}
\varepsilon_{\ast}(S^m(E_{\ast}),\,x)\,\,=\, \,N(E_{\ast})\cdot \inf\limits_{C\subset Y}
\left\{\frac{\mu_{\min}\bigl(\nu^{\ast}(S^mE')\bigr)}{\sum\limits_{y\in
\gamma^{-1}(x)}\mult_y C}\right\},
\end{align*}
where the infimum is taken over all irreducible curves $C\,\subset\, Y$ such that
$C\bigcap \gamma^{-1}(x) \,\neq \,\emptyset$, and $\nu \,:\, \overline{C}
\,\longrightarrow\, C$ is the normalization map.
 
As taking symmetric powers commutes with the operation of pullback, and $$\mu_{\min}(S^mV)
\,=\, m\mu_{\min}(V)$$ for a vector bundle $V$ on a smooth complex projective curve $\overline{C}$, we have
\begin{align*}
\varepsilon_{\ast}(S^m(E_{\ast}),\,x) \,= \,N(E_{\ast})\cdot \inf\limits_{C\subset Y}
\left\{\frac{m\mu_{\min}\bigl(\nu^{\ast}E'\bigr)}{\sum\limits_{y\in
\gamma^{-1}(x)}\mult_y C}\right\} \,=\, m\varepsilon_{\ast}(E_{\ast},\,x).
\end{align*}
This completes the proof.
\end{proof}

\begin{thm}\label{thm-tensor}
Let $E_{\ast}$ and $F_{\ast}$ be two parabolic nef vector bundles on a smooth irreducible
complex projective variety $X$ having a common parabolic divisor $D\,\subset\, X$. Then for
every point $x\,\in\, X$,
\begin{align*}
\varepsilon_{\ast}(E_{\ast}\otimes F_{\ast},\,x)\,\,\geq\,\, N(E_{\ast}\otimes F_{\ast})
\cdot\left\{\frac{\varepsilon_{\ast}(E_{\ast},\,x)}{N(E_{\ast})}\,+\,
\frac{\varepsilon_{\ast}(F_{\ast},\,x)}{N(F_{\ast})}\right\}.
\end{align*}
\end{thm}

\begin{proof}
The correspondence between parabolic vector bundles and orbifold bundles takes tensor 
product of two orbifold bundles to the tensor product of the corresponding parabolic vector 
bundles (see \cite{B97I}). Let $E'$ and $F'$ be the orbifold bundles on $Y$ corresponding
to $E_*$ and $F_*$ respectively, where $Y$ is the 
Galois cover of $X$ as in Section \ref{orbifold}. Then $E'\otimes F'$ corresponds to the 
parabolic bundle $E_{\ast}\otimes F_{\ast}$.
By Theorem \ref{thm3.2}, for any point $x\,\in\, X$, we have
$$\varepsilon_{\ast}(E_{\ast}\otimes F_{\ast},\,x) \,\,=\,\, N(E_{\ast}\otimes F_{\ast})
\cdot\inf\limits_{C\subset Y} \left\{\frac{\mu_{\min}\bigl(\nu^{\ast}(E'\otimes F')\bigr)}{\sum\limits_{y\in \gamma^{-1}(x)}\mult_y C}\right\},
$$
where the infimum is taken over all irreducible curves $C\,\subset\, Y$ such that $C\bigcap 
\gamma^{-1}(x)\,\neq\, \emptyset$, and $\nu \,:\, \overline{C} \,\longrightarrow\, C$ is the 
normalization map (see Theorem \ref{thm3.2}).
For the smooth curve $\overline{C}$, we also have
$$\mu_{\min}\bigl(\nu^{\ast}(E'\otimes F')\bigr)\, =\, \mu_{\min}(\nu^{\ast}E') + \mu_{\min}(\nu^{\ast}F').$$
Therefore,
$$\varepsilon_{\ast}(E_{\ast}\otimes F_{\ast},\,x) \,=\, N(E_{\ast}\otimes F_{\ast})\cdot
\inf\limits_{C\subset Y} \left\{\frac{\mu_{\min}(\nu^{\ast}E')}{\sum\limits_{y\in
\gamma^{-1}(x)}\mult_y C} \,+\, \frac{\mu_{\min}(\nu^{\ast}F')}{\sum\limits_{y\in
\gamma^{-1}(x)}\mult_y C}\right\}$$
$$ \geq \, N(E_{\ast}\otimes F_{\ast})\cdot \left\{\frac{\varepsilon_{\ast}(E_{\ast},\,
x)}{N(E_{\ast})}\,+\,\frac{\varepsilon_{\ast}(F_{\ast},\,x)}{N(F_{\ast})}\right\}.$$ 
This completes the proof.
\end{proof}

\subsection{Examples and further questions}\label{examples}

Let $X$ be a smooth complex projective variety. Let $E_{\ast}$ be a parabolic vector bundle 
on $X$ with $E$ as the underlying vector bundle. If $X$ has dimension one, and $E$ is ample,
then it 
follows that $E_{\ast}$ is parabolic ample; see \cite[Theorem 3.1]{B97I}. But in general 
this is not the case. We will give two examples to show that the parabolic ampleness of 
$E_{\ast}$ does not imply the ampleness of $E$ and vice versa.

Let $D$ be an effective divisor on $X$ and suppose that 
$D\,=\, \sum\limits_{i=1}^nD_i$ is the decomposition of $D$ into irreducible components.
A parabolic line bundle on $X$ with parabolic divisor $D$ is a pair $L_{\ast} \,=\,
\bigl( L,\,\{\alpha_1,\,\cdots,\,\alpha_n\}\bigr)$, where
$L$ is a line bundle on $X$ and the parabolic weight $0\,\leq\,\alpha_i\,<\,1$ corresponds to the divisor $D_i$. 
Assume that $\alpha_i \,\in\, \mathbb{Q}$ for all $i$. Then $L_{\ast}$ is parabolic ample if
and only if 
\begin{align*}
c_1(L_{\ast}) \,:=\, c_1(L) + \sum\limits_{i=1}^n \alpha_i[D_i] \,\in\, H^{1,1}(X)
\cap H^2(X,\mathbb{Q})
\end{align*}
lies in the ample cone of $X$. 

\begin{exm}\label{example1}
Let $\pi \,: \,\mathbb{P}_C(W)\,\longrightarrow \,C$ be a ruled surface over an elliptic curve
$C$ defined by the normalized rank 2 bundle $W$ which sits in the following exact sequence:
$$0\,\longrightarrow\, \mathcal{O}_C \,\longrightarrow\, W \,\longrightarrow
\,\mathcal{O}_C(p) \,\longrightarrow \,0$$ for some point $p\,\in\, C$. Therefore the
associated invariant $e \,=\, -\deg(W)$ is $-1$ (see \cite[Proposition 2.8]{Har1}).
We consider the simple normal crossing divisor $\sigma+f$, where $\sigma$ denotes the 
normalized section of $\pi$ such that $\mathcal{O}_{\mathbb{P}(W)}(\sigma)\, \simeq\,
\mathcal{O}_{\mathbb{P}(W)}(1)$ and $f$ denotes a fiber of the map $\pi$.
 
Let $L$ be a line bundle on $\mathbb{P}_C(W)$ which is numerically equivalent to $\frac{1}{2}\sigma - \frac{1}{2} f$. Then $L$ is not ample by 
\cite[Proposition 2.21]{Har1}.
Next consider the parabolic bundle $$L_* \,= \,(L,\{\alpha_1,\,\alpha_2\})$$
on $\mathbb{P}_C(W)$ with parabolic weights $\alpha_1 \,=
\, \frac{1}{2}$ and $\alpha_2 \,=\,\frac{1}{2}$. 
Note that $$c_1(L_*) \,=\, c_1(L) + \alpha_1\sigma +\alpha_2 f \,=\, \sigma$$
which is inside the ample cone of $\mathbb{P}_C(W)$.
Thus $L_*$ is parabolic ample, but $L$ is not ample. 
\end{exm}

\begin{exm}\label{example2} 
Let $\pi \,:\, \mathbb{P}_C(W)\,\longrightarrow\, C$ be a ruled surface over 
$\mathbb{P}^1_{\mathbb{C}}$ defined by the normalized rank 2 bundle $W \,=\, 
\mathcal{O}_{\mathbb{P}^1_{\mathbb{C}}}\oplus \mathcal{O}_{\mathbb{P}^1_{\mathbb{C}}}(-1)$. 
Then the associated invariant $e$ is equal to $1$. As done above, we consider the simple normal 
crossing divisor $\sigma+f$, where $\sigma$ denotes the normalized section of $\pi$ such 
that $\mathcal{O}_{\mathbb{P}(W)}(\sigma) \,\simeq\, \mathcal{O}_{\mathbb{P}(W)}(1)$ and $f$ 
denotes a fiber of the map $\pi$.
 
Let $L\,\equiv\, \frac{1}{2}\sigma + f$ be a line bundle on $\mathbb{P}_C(W)$. Then $L$ is 
ample by \cite[Proposition 2.20]{Har1}. Next consider the parabolic bundle $$L_* \,=\,
(L,\,\{\alpha_1,\,\alpha_2\})$$ with parabolic weights $\alpha_1 \,
=\, \frac{7}{8}$ and $\alpha_2 \,=\,\frac{1}{4}$.
Note that $$c_1(L_*) \,=\, c_1(L) + \alpha_1\sigma +\alpha_2 f \,=\,
\frac{11}{8}\sigma + \frac{10}{8}f$$ which does not lie in the ample cone of $\mathbb{P}_C(W)$.
Thus $L_*$ is not parabolic ample, but $L$ is ample.
\end{exm}

In view of Example \ref{example1} and Example \ref{example2}, using the Seshadri criterion for 
parabolic ampleness (see Theorem \ref{seshadri-crit}) we see that there can not be an 
inequality between $\varepsilon_{\ast}(L_{\ast},\,x)$ and $\varepsilon(L,\,x)$, in general. 
It is still interesting to ask the following question.

\begin{question}\label{quest1}
Suppose $E_{\ast}$ is an ample parabolic bundle on a projective variety $X$ such that the underlying vector
bundle $E$ is also ample. Then can we compare $\varepsilon_{\ast}(E_{\ast},\,x)$ and $\varepsilon(E,\,x)$?
\end{question}

\begin{exm}
Let $Z$ be a smooth complex projective variety, and let $W$ be a vector bundle on $Z$. 
Let $$\rho\,:\,X \,:=\, \mathbb{P}(W)\,\longrightarrow\, Z$$ be the associated projective bundle on $Z$. 
Take $D \,\subset\, Z$ to be a normal crossing divisor on $Z$, and let $F_*$ be a semistable 
parabolic bundle of rank $r$ on $Z$ with parabolic divisor $D$. Then $\rho^*F_*$ is a 
parabolic semistable bundle on $X$ with parabolic divisor $D' \,=\, \rho^*(D)$. Let $D' \,=\, 
\sum\limits_{i=1}^nD'_i$ be the decomposition of $D'$ into irreducible components.
 
 Let $L_{\ast}$ be a parabolic line bundle with parabolic divisor $D'$ on $X$. So
$L_{\ast}$ is given by a pair
$\bigl( L,\,\{\alpha_1,\,\cdots,\,\alpha_n\}\bigr)$, where
$L$ is a line bundle on $X$ and $0\,\leq\,\alpha_i\,<\,1$ corresponds to the divisor $D'_i$.
Assume that $\alpha_i \,\in\, \mathbb{Q}$ for all $i$. Then
$$E_{\ast}\,=\, \rho^{\ast}(F_{\ast}) \otimes L_{\ast}$$ is parabolic semistable with
$\triangle_{par}(E_{\ast}) \,=\, 0$. Note that
$$c_1(L_{\ast}) \,:=\, c_1(L) + \sum\limits_{i=1}^n \alpha_i[D_i].$$ We may choose
$L_{\ast}$ in such a way that $$c_1(E_{\ast}) \,=\, c_1(\rho^{\ast}F_{\ast}) + r c_1(L_{\ast})$$
lies in the ample cone of $X$. This way we can produce parabolic ample bundles on $X$ with
$\triangle_{par}(E_{\ast}) \,=\, 0$ (see \cite[Example 22]{MR21}). Using Remark
\ref{par-disc} we can compute the parabolic Seshadri constants for such parabolic bundles.
\end{exm}

In this situation, there is the following result on the relationship between the
parabolic ampleness of $E_{\ast}$ and the ampleness of its first parabolic Chern class.

\begin{thm}[{\cite[Theorem 19]{MR21}}]\label{thm5.1}\label{mr}
Let $E_{\ast}$ be a semistable parabolic vector bundle of rank $r$ on a smooth complex 
projective variety $X$ such that $\triangle_{par}(E_{\ast}) \,=\, 0$. Then $E_{\ast}$ is 
parabolic ample if and only if its parabolic first Chern class $c_1(E_{\ast})$ is in the 
ample cone of $X$.
\end{thm}

So one can ask the following question. 

\begin{question}\label{quest2}
In the situation of Theorem \ref{mr}, can we compare $\varepsilon_{\ast}(E_{\ast},\,x)$ and 
$\varepsilon\bigl(c_1(E_{\ast}),\,x\bigr)$?
\end{question}

\section*{Acknowledgements}
We are grateful to the referee for carefully reading the paper and giving several useful suggestions
which improved the paper.
The second and the fourth authors are partially supported by a grant from Infosys Foundation. 
The third author is supported financially by SERB-NPDF fellowship (File no: PDF/2021/000282).
The first author is partially supported by a J. C. Bose Fellowship.

\end{document}